\documentclass{article}

\usepackage{hyperref}
\usepackage{amsmath}
\usepackage{amssymb}
\usepackage{amsthm}

\newtheorem{lemma}{Lemma}
\newtheorem{theorem}{Theorem}

\newcommand{\dotoplus}{ \mathbin{\dot{\oplus}} }
\newcommand
{\sMat}
[4]
{
	\bigl(
	\begin{smallmatrix}
		#1 \mathstrut
		&
		#2 \mathstrut
		\\
		#3 \mathstrut
		&
		#4 \mathstrut
	\end{smallmatrix}
	\bigr)
}

\DeclareMathOperator{\mat}{M}
\DeclareMathOperator{\Cent}{C}
\DeclareMathOperator{\Ker}{Ker}
\DeclareMathOperator{\Aut}{Aut}
\DeclareMathOperator{\Spec}{Spec}
\DeclareMathOperator{\Hom}{Hom}
\DeclareMathOperator{\Norm}{N}
\DeclareMathOperator{\Trace}{T}
\DeclareMathOperator{\Bimul}{U}
\newcommand{\e}{\mathrm{e}}
\newcommand
{\inv}
[1]
{
	\!\;
	\overline{
		\!\!\: #1 \vphantom{!} \!\!\:
	}
	\;\!
}

\newcommand{\up}[2]{ { ^{#1} \! {#2} } }
\newcommand
{\TitsIndex}
[5]
{
	\up{#2}{ \mathsf{#1}_{#3, #4}^{#5} }
}

\title{
	Weyl elements in isotropic reductive groups
}

\author{
	Egor Voronetsky%
	\thanks{
		This work was performed at the Saint Petersburg Leonhard Euler International Mathematical Institute and supported by the Ministry of Science and Higher Education of the Russian Federation (agreement no. 075--15--2025--343).
	} \\
	Saint Petersburg University, \\
	7/9 Universitetskaya nab., \\
	St. Petersburg, 199034 Russia
}

\begin{document}
	\maketitle
	
	\begin{abstract}
		We study Weyl elements in isotropic reductive groups over commutative rings. Our main result in an explicit formula for squares of such elements. We also classify these elements in rank one groups and prove basic properties of their loci.
	\end{abstract}

	\section{Introduction}
	
	Throughout this paper
	\( \Phi \) denotes an irreducible crystallographic root system, not necessarily reduced. In other words, the type of
	\( \Phi \) is one of
	\( \mathsf{A}_\ell \),
	\( \mathsf{B}_\ell \),
	\( \mathsf{C}_\ell \),
	\( \mathsf{BC}_\ell \),
	\( \mathsf{D}_\ell \),
	\( \mathsf{E}_6 \),
	\( \mathsf{E}_7 \),
	\( \mathsf{E}_8 \),
	\( \mathsf{F}_4 \),
	\( \mathsf{G}_2 \).
	
	A group
	\( G \) is called
	\( \Phi \)-graded for a root system
	\( \Phi \) if there are
	\textit{root subgroups}
	\( U_\alpha \leq G \) for
	\( \alpha \in \Phi \) such that
	\begin{itemize}
		
		\item
		\( U_{2 \alpha} \leq U_\alpha \) if both
		\( \alpha \) and
		\( 2 \alpha \) are roots (so the type of
		\( \Phi \) is
		\( \mathsf{BC}_\ell \));
		
		\item
		\(
		[ U_\alpha, U_\beta ]
		\leq
		\bigl\langle
		U_{ i \alpha + j \beta }
		\mid
		i \alpha + j \beta \in \Phi;\,
		i, j > 0
		\bigr\rangle
		\) for non-antiparallel
		\( \alpha \) and
		\( \beta \);
		
		\item
		if
		\( \Phi^{+} \subset \Phi \) is a subset of positive roots, then the multiplication map
		\(
		\prod_{ \alpha \in \Phi^{+} \setminus 2 \Phi^{+} }
		U_\alpha
		\to
		\bigl\langle
		U_\alpha
		\mid
		\alpha \in \Phi^{+}
		\bigr\rangle
		\) is one-to-one;
		
		\item
		for all
		\( \alpha \in \Phi \) there exist
		\textit{%
			\( \alpha \)-Weyl elements%
		}.
		
	\end{itemize}
	
	In this setting an element
	\( w \in U_\alpha U_{- \alpha} U_\alpha \) is called
	\( \alpha \)-Weyl if
	\( \up{w}{U_\beta} = U_{ s_\alpha(\beta) } \) for all
	\( \beta \), where
	\(
	s_\alpha(\beta)
	=
	\beta
	-
	2 \alpha (\alpha \cdot \beta) / (\alpha \cdot \alpha)
	\) is the Euclidean reflection determined by
	\( \alpha \).
	
	There is also a slightly more general notion of root gradings by not necessarily crystallographic root systems
	\cite{%
		muhl-weiss,
		wiedemann%
	}.
	
	There is a large class of root graded groups coming from algebraic geometry. Namely, let
	\( K \) be a non-zero unital commutative ring and
	\( G \) be an isotropic (in a suitable sense) reductive group scheme over
	\( K \). For simplicity we assume that the root datum of the split form of
	\( G \) is constant on
	\( \Spec(K) \) and the corresponding
	\textit{absolute} root system
	\( \widetilde{\Phi} \) is irreducible. Then the point group
	\( G(K) \) is root graded by the
	\textit{relative} root system
	\( \Phi \), we call it an
	\textit{isotropic reductive group}. The class of all such groups contains finite simple groups of Lie type (up to isogeny and excluding Suzuki and Ree groups, though they are also root graded), non-compact real reductive Lie groups, and Chevalley groups over arbitrary rings such as
	\( \mathrm{SL}(n, K) \) and
	\( \mathrm{SO}(2 n, K) \).
	
	We are interested in description and properties of Weyl elements in an isotropic reductive group
	\( G(K) \). Such elements are well studied for Chevalley groups
	\( \mathrm{G}(\Phi, K) \) and the corresponding Steinberg groups
	\( \mathrm{St}(\Phi, K) \), see e.g.
	\cite{allcock}. If the rank of
	\( \Phi \) is at least
	\( 3 \) or the base ring
	\( K \) is sufficiently close to a field, then there is a general classification of all
	\( \Phi \)-graded groups
	\cite{%
		h3-h4,
		muhl-weiss,
		h2,
		tits-polygons,
		shi,
		timmesfeld,
		root-graded,
		wiedemann,
		zhang%
	}, including an explicit description of Weyl elements.
	
	Our work solves an open problem appearing in the classification of isomorphisms between isotropic reductive groups
	\cite[remark 4.4]{gvozdevsky}. Namely, we prove the
	\textit{square formula} (theorem
	\ref{squares}) in isotropic reductive groups,
	\[
	\up{w^2}{u}
	=
	\begin{cases}
		u,
		& 2 (\alpha \cdot \beta) / (\alpha \cdot \alpha)
		\text{ is even},
		\\
		u^{- 1},
		& 2 (\alpha \cdot \beta) / (\alpha \cdot \alpha)
		\text{ is odd and } 2 \beta \notin \Phi,
		\\
		u \cdot (- 1),
		& 2 (\alpha \cdot \beta) / (\alpha \cdot \alpha)
		\text{ is odd and } 2 \beta \in \Phi
	\end{cases}
	\]
	for
	\( \alpha \)-Weyl elements
	\( w \) and all elements
	\( u \in U_\beta \), where
	\( 2 \alpha \notin \Phi \) and
	\( u \cdot (- 1) \) is a suitable operation defined below. The square formula is already known for groups of rank at least
	\( 3 \) and in the simply laced case by
	\cite[propositions 5.4.15, 7.6.15, 9.5.13]{wiedemann}. We also describe all
	\( \alpha \)-Weyl elements for
	\( 2 \alpha \notin \Phi \), prove their existence (locally in Zariski topology) and some basic properties in theorem
	\ref{weyl}, and explicitly calculate their squares in the main theorem
	\ref{squares}. Moreover, we find normalizers of long root subgroups in the previously unknown rank
	\( 1 \) case in theorem
	\ref{long-norm} and construct the canonical split torus with a root datum in theorem
	\ref{split-torus} and lemma
	\ref{coroots}, this is an analogue of a maximal split tori in reductive group schemes over local rings.
	
	It is also worth to recall that if
	\( \alpha \) and
	\( \beta \) are two basic roots with
	\( \angle(\alpha, \beta) = \pi - \pi / m \) and
	\( w_\alpha \),
	\( w_\beta \) are corresponding Weyl elements in an abstract root graded group, then the
	\textit{braid relations} hold
	\cite[theorem 2.2.34]{wiedemann}, i.e.
	\[
	w_\alpha w_\beta w_\alpha \cdots
	=
	w_\beta w_\alpha w_\beta \cdots,
	\]
	where both sides contain
	\( m \) factors. In particular, if
	\( m = 2 \), then these Weyl elements commute. Together with the square formula this allows us to define the
	\text{extended Weyl group} as in the case of Chevalley groups, i.e. an abstract group with the generators
	\( w_\alpha \) for basic
	\( \alpha \) and only braid relations and
	\( \up{ w_\alpha^2 }{ w_\beta } = w_\beta^{\pm 1} \). Another important property is that if
	\( w = g_1 g_2 g_3 \) is an
	\( \alpha \)-Weyl element for
	\( g_1, g_2 \in U_\alpha \) and
	\( g_2 \in U_{- \alpha} \), then
	\( w^{- 1} = g_3^{- 1} g_2^{- 1} g^{- 1} \) is an
	\( \alpha \)-Weyl element and
	\( w = g_2 g_3 g_1^w = \up{w}{g_3} g_1 g_2 \) is also an
	\( (- \alpha) \)-element
	\cite[proposition 2.2.6]{wiedemann}.
	
	Along the way we also generalize a construction of the group scheme
	\( \mathbb{E}_7^{\mathrm{sc}} \) using Albert algebras from the field case
	\cite{muh-wei-e7} to arbitrary commutative rings. In principle one can also work with this group scheme using its abstract
	\( 56 \)-dimensional presentation as in
	\cite{vav-luz}, but this leads to cumbersome formulae involving large matrices. There is also a construction of
	\( \mathbb{E}_7^{\mathrm{sc}} \) if
	\( 2 \) is invertible using Brown algebras
	\cite{%
		alsaody,
		gar-pet-rac%
	}.

	\section{Isotropic reductive groups}
	
	Let
	\( K \) be a non-zero unital commutative ring and
	\( \widetilde{\Phi} \) be a reduced root system. Choose a weight lattice
	\[
	\textstyle
	\mathbb{Z} \widetilde{\Phi}
	\leq
	\Lambda
	\leq
	\bigl\{
	x
	\mid
	\forall \alpha \in \widetilde{\Phi} \enskip
	2 (x \cdot \alpha) / (\alpha \cdot \alpha)
	\in
	\mathbb{Z}
	\bigr\}
	\]
	and consider the Chevalley--Demazure group scheme
	\( \mathrm{G}^\Lambda(\widetilde{\Phi}, {-}) \) over
	\( K \). We denote its standard maximal torus by
	\( \widetilde{T} \) and its standard root subgroups by
	\( \widetilde{U}_\alpha \) for
	\( \alpha \in \widetilde{\Phi} \). Now choose a surjective map
	\(
	u
	\colon
	\widetilde{\Phi} \sqcup \{ 0 \}
	\to
	\Phi \sqcup \{ 0 \}
	\) arising from one of irreducible Tits indices
	\cite[table II]{index-field} and let
	\[
	L
	=
	\bigl\langle
	\widetilde{T}, \widetilde{U}_\alpha
	\mid
	u(\alpha) = 0
	\bigr\rangle,
	\quad
	U_\alpha
	=
	\bigl\langle
	\widetilde{U}_{\alpha'}
	\mid
	u(\alpha') \in \{ \alpha, 2 \alpha \}
	\bigr\rangle
	\]
	be group subsheaves of
	\( \mathrm{G}^\Lambda(\widetilde{\Phi}, {-}) \), they are closed subschemes. More precisely,
	\( L \) is split reductive with a maximal torus
	\( \widetilde{T} \) and
	\( U_\alpha \) are unipotent in a suitable sense.
	
	In this paper we say that a group scheme
	\( G \) together with a family of group subschemes
	\( L, U_\alpha \leq G \) for
	\( \alpha \in \Phi \) is an
	\textit{isotropic reductive group scheme} if fppf locally it is isomorphic to
	\( \mathrm{G}^\Lambda(\widetilde{\Phi}, {-}) \) with the standard group subschemes constructed by a map
	\(
	u
	\colon
	\widetilde{\Phi} \sqcup \{ 0 \}
	\to
	\Phi \sqcup \{ 0 \}
	\) as above. An isotropic reductive group scheme
	\( G \)
	\textit{splits} if it is isomorphic to
	\( \mathrm{G}^\Lambda(\widetilde{\Phi}, {-}) \) with the standard group subschemes.
	
	For example, if
	\( A \) is an Azumaya algebra of rank
	\( d \geq 1 \) over
	\( K \), then
	\( \mathrm{SL}(r + 1, A) \) is the point group of an isotropic reductive group scheme over
	\( K \) with the absolute root system of type
	\( \mathsf{A}_{(r + 1) d - 1} \) and the relative root system of type
	\( \mathsf{A}_r \). This group consists of
	\( (r + 1) \times (r + 1) \) matrices with trivial reduced norm in the Azumaya algebra
	\( \mat(r + 1, A) \).
	
	This class of group schemes is interesting because every simple reductive group scheme
	\( G \) over a local ring
	\( K \) is either anisotropic or isotropic in the above sense in a canonical way up to conjugation, where
	\[
	L
	\rtimes
	\bigl\langle
	U_\alpha
	\mid
	\alpha \in \Phi^{+}
	\bigr\rangle
	\]
	is a minimal parabolic subgroup, see
	\cite[XXVI]{sga3} and
	\cite{index-loc} for details. Moreover, if
	\( G \) is a locally isotropic simple reductive group scheme over arbitrary ring
	\( K \) and
	\( \mathfrak{p} \trianglelefteq K \) is a prime ideal, then the isotropic structure of
	\( G_{\mathfrak{p}} \) admits a continuation to
	\( G_s \) for sufficiently large
	\( s \in K \setminus \mathfrak{p} \). Finally, many questions on generic reductive group schemes reduce to simple ones using direct product decomposition and Weil restriction
	\cite[XXIV, proposition 5.10]{sga3}, see lemma
	\ref{ssys-sgr} below for an example of such approach.
	
	In any isotropic reductive group scheme
	\( G \) both
	\( G \) and
	\( L \) are indeed reductive group schemes, i.e. they are affine, smooth, and their geometric fibers are (connected) reductive algebraic groups. We say that
	\( U_\alpha \) are the
	\textit{root subgroups} of
	\( G \). Clearly,
	\( U_{2 \alpha} \leq U_\alpha \) for
	\( \alpha, 2 \alpha \in \Phi \) (if the type of
	\( \Phi \) is
	\( \mathsf{BC}_\ell \)),
	\[
	[U_\alpha, U_\beta]
	\leq
	\bigl\langle
	U_{i \alpha + j \beta}
	\mid
	i \alpha + j \beta \in \Phi,\,
	i, j > 0
	\bigr\rangle
	\]
	for non-antiparallel
	\( \alpha \) and
	\( \beta \), for any subset of positive roots
	\( \Phi^{+} \) the multiplication map
	\[
	\prod_{\alpha \in \Phi^{\pm} \setminus 2 \Phi^{\pm}}
	U_\alpha
	\to
	U^{\pm}
	=
	\langle U_\alpha \mid \alpha \in \Phi^{\pm} \rangle
	\]
	is one-to-one,
	\( L \) normalizes all
	\( U_\alpha \), and the multiplication map
	\[
	U^{-} \times L \times U^{+} \to G
	\]
	is an open embedding. The group subsheaves
	\( L \rtimes U^{\pm} \) are opposite parabolic subgroups,
	\( U^{\pm} \) are their unipotent radicals, and
	\( L \) is their common Levi subgroup.
	
	If
	\( \pi \colon G \to G' \) is a central isogeny of twisted forms of Chevalley--Demazure group schemes, then
	\( G \) is isotropic if and only if
	\( G' \) is isotropic. More precisely, if
	\( G \) is isotropic, then the images of its distinguished group subschemes form an isotropic structure on
	\( G' \). Conversely, if
	\( G' \) is isotropic with distinguished group subschemes
	\( L \) and
	\( U_\alpha \), then
	\( G \) is isotropic with
	\( \pi^{- 1}(L) \) and the scheme commutators
	\( [ \pi^{- 1}(L), \pi^{- 1}(U_\alpha) ] \).
	
	Let
	\( \mathfrak{g} \) be the Lie
	\( K \)-algebra of
	\( G \). It has a weight decomposition
	\[
	\mathfrak{g}
	=
	\mathfrak{g}_0
	\oplus
	\bigoplus_{\alpha \in \Phi}
	\mathfrak{g}_\alpha,
	\]
	where
	\( \mathfrak{g}_0 \) is the Lie algebra of
	\( L \),
	\( \mathfrak{g}_\alpha \) is the Lie algebra of
	\( U_\alpha \) for
	\( 2 \alpha \notin \Phi \), and
	\( \mathfrak{g}_\alpha \oplus \mathfrak{g}_{2 \alpha} \) is the Lie algebra of
	\( U_\alpha \) for
	\( 2 \alpha \in \Phi \). If
	\( 2 \alpha \notin \Phi \), then there is a canonical isomorphism
	\(
	t_\alpha
	\colon
	\mathfrak{g}_\alpha \otimes_K ({-})
	\to
	U_\alpha
	\), so
	\( U_\alpha(K) \) has a structure of a
	\( K \)-module. Otherwise
	\( \Phi \) is of type
	\( \mathsf{BC}_\ell \) and there is a canonical isomorphism
	\cite[\S 2]{iso-elem}
	\[
	P_\alpha \boxtimes_K ({-}) \to U_\alpha,
	\]
	where
	\( P_\alpha \) is a so-called
	\textit{%
		\( 2 \)-step nilpotent
		\( K \)-module%
	}
	\cite[\S 4]{twisted-forms} with the distinguished short exact sequence
	\[
	\dot 0
	\to
	\mathfrak{g}_{2 \alpha}
	\to
	P_\alpha
	\to
	\mathfrak{g}_\alpha
	\to
	\dot 0,
	\]
	and
	\( \boxtimes_K \) denotes the scalar extension of these objects. In particular, there is a right action
	\( (x, k) \mapsto x \cdot k \) of the multiplicative monoid
	\( K^\bullet \) on the group
	\( P_\alpha \). It is easy to see that the above exact sequence splits in a non-canonical way, i.e.
	\(
	P_\alpha
	\cong
	\mathfrak{g}_{2 \alpha} \dotoplus \mathfrak{g}_\alpha
	\), where
	\[
	(x \dotoplus y) \dotplus (x' \dotoplus y')
	=
	(x + c(y, y') + x') \dotoplus (y + y'),
	\quad
	(x \dotoplus y) \cdot k = x k^2 \dotoplus y k
	\]
	for a suitable bilinear map
	\(
	c
	\colon
	\mathfrak{g}_\alpha \times \mathfrak{g}_\alpha
	\to
	\mathfrak{g}_{2 \alpha}
	\). Indeed, every
	\( 2 \)-step nilpotent module
	\( (M, M_0) \) with projective factor-module
	\( M / M_0 \) splits. This is clear if
	\( M / M_0 \) is free, and the general case follows by adding a direct summand complementing
	\( M / M_0 \) to a free module. It is convenient to denote
	\( \mathfrak{g}_\alpha \) by
	\( P_\alpha \) if
	\( 2 \alpha \notin \Phi \), so in all cases
	\( U_\alpha(K) \cong P_\alpha \).
	
	\begin{lemma}
		\label{aut}
		Let
		\( G \) be isotropic reductive group scheme over
		\( K \). Then its automorphism group scheme
		\( A = \mathbf{Aut}(G, L, U_\alpha) \) is an extension of a finite group scheme (a subgroup of the constant group sheaf
		\( \mathbf{Out}(\widetilde{\Phi}) \)) by
		\( L / \Cent(G) \).
	\end{lemma}
	\begin{proof}
		We can assume that
		\( G \) splits and the spectrum of
		\( K \) is connected (and, if necessary, that
		\( K \) is Noetherian). By definition,
		\[
		A(K)
		\leq
		\Aut(G)
		\cong
		(G / \Cent(G))(K) \rtimes O
		\]
		consists of automorphisms
		\( f \) with the properties
		\( f(L) = L \) and
		\( f(U_\alpha) = U_\alpha \) for all
		\( \alpha \in \Phi \), where
		\( O \leq \mathrm{Out}(\widetilde{\Phi}) \) is the stabilizer of the weight lattice
		\( \Lambda \). We can assume that
		\( O \) preserves the maximal torus
		\( \widetilde{T} \) and permutes the root subgroups
		\( \widetilde{U}_\alpha \). Moreover, if we fix subsets of positive roots
		\( \widetilde{\Phi}^{+} \subset \widetilde{\Phi} \) and
		\( \Phi^{+} \subset \Phi \) in such a way that
		\(
		u(\widetilde{\Phi}^{+} \cup \{ 0 \})
		=
		\Phi^{+} \cup \{ 0 \}
		\), then we can assume that
		\( O \) preserves the Borel subgroup
		\[
		\widetilde{T}
		\rtimes
		\bigl\langle
		\widetilde{U}_\alpha
		\mid
		\alpha \in \widetilde{\Phi}^{+}
		\bigr\rangle.
		\]
		These properties determine liftings of outer automorphisms to ordinary automorphisms up to a choice of structure constants.
		
		Clearly,
		\( L / \Cent(G) \trianglelefteq A \) is a normal group subscheme. If
		\( f \in A(K) \), then multiplying it fppf locally by elements of
		\( L / \Cent(G) \) we can assume that
		\( f(\widetilde{T}) = \widetilde{T} \) and
		\( f \) preserves the set
		\( \widetilde{\Phi}^{+} \cap u^{- 1}(0) \) of positive roots of the split reductive group scheme
		\( L \). Since
		\( f(U_\alpha) = U_\alpha \), it also preserves all sets
		\( u^{- 1}(\alpha) \subset \widetilde{\Phi} \) for
		\( \alpha \in \Phi \), so
		\( f \) preserves
		\( \Phi^{+} \) and lies in
		\( (\widetilde{T} / \Cent(G))(K) \rtimes O \). The first factor is contained in
		\( (L / \Cent(G))(K) \), so without loss of generality
		\( f \in O \). It remains to note that the intersection
		\( O \cap A(K) \) is independent of the choice of the lifting of
		\( O \) to a subgroup of
		\( \Aut(G) \) preserving
		\( \widetilde{T} \) and the Borel subgroup, i.e. it depends only on the Tits index.
	\end{proof}
	
	It turns out that every isotropic reductive group scheme
	\( G \) has a canonical split torus
	\( T \leq L \) of rank
	\( \mathrm{rk}(\Phi) \). We construct this torus in the simply connected case, and in remaining cases one can take the corresponding factor-group scheme.
	
	\begin{theorem}
		\label{split-torus}
		Let
		\( G \) be a simply connected isotropic reductive group scheme over
		\( K \) and
		\( H \) be the scheme centralizer of the automorphism group scheme
		\( A = \mathrm{Aut}(G, L, U_\alpha) \) in
		\( G \). Then
		\( H \) is a group scheme of muliplicative type, its largest subtorus
		\( T \leq H \) has rank
		\( \mathrm{rk}(\Phi) \), and
		\( \Phi \) canonically embeds into its character group. In particular,
		\( T \) splits. Each
		\( g \in T(K) \) acts on root subgroups by the formulae
		\(
		\up{g}{ t_\alpha(x) }
		=
		t_\alpha(x\, \alpha(g))
		\) for
		\( 2 \alpha \notin \Phi \) and
		\(
		\up{g}{ t_\alpha(x) }
		=
		t_\alpha(x \cdot \alpha(g))
		\) otherwise.
	\end{theorem}
	\begin{proof}
		We again can assume that
		\( G \) splits and the spectrum of
		\( K \) is connected. Let
		\( O = \mathrm{Out}(\widetilde{\Phi}) \) be the outer automorphism group of
		\( G \) lifted to a subgroup of
		\( A(K) \) preserving a maximal torus and the Borel subgroup corresponding to a choice of subsets of positive roots
		\( \widetilde{\Phi}^{+} \subseteq \widetilde{\Phi} \),
		\( \Phi^{+} \subseteq \Phi \) as in the proof of lemma
		\ref{aut}. Let also
		\( O_{L, U_\alpha} \leq O \) be the stabilizer of the isotropic structure, i.e. if we identify
		\( O \) with the group of automorphisms of
		\( \Phi \) preserving
		\( \widetilde{\Phi} \), then
		\( O_{L, U_\alpha} \) consists of automorphisms
		\( f \) with the property
		\( u \circ f = u \). Clearly,
		\( H \leq \Cent(L) \leq \widetilde{T} \) is the centralizer of
		\( O_{L, U_\alpha} \) (the maximal torus
		\( \widetilde{T} \) is its own centralizer in
		\( G \) by
		\cite[XIX, 2.8]{sga3}), so it is a group scheme of multiplicative type.
		
		Recall that any irreducible Tits index consists of a subgroup
		\( \Gamma \leq O \) (possibly a proper subgroup of
		\( O_{L, U_\alpha} \)) and a
		\( \Gamma \)-invariant set
		\( J \) of vertices of the Dynkin diagram of
		\( \widetilde{\Phi} \). The map
		\( u \) is induced by a linear map between the corresponding Euclidean spaces. This linear map is obtained by taking the factor-space by all basic roots not in
		\( J \) and by identifying roots from
		\( J \) differing by the action of
		\( \Gamma \). Clearly, the corresponding subtorus of
		\( \widetilde{T} \) is precisely
		\( T \), so
		\( \Phi \) is a subset of its character group, the inclusion
		\( T \leq \widetilde{T} \) induces the map
		\(
		u
		\colon
		\widetilde{\Phi} \sqcup \{ 0 \}
		\to
		\Phi \sqcup \{ 0 \}
		\), and the action of
		\( T \) on root subgroups is given by the formulae from the statement.
	\end{proof}
	
	It is easy to see that for any choice of the weight lattice the group subscheme
	\( L \) is the scheme centralizer of
	\( T \). Indeed, such centralizer is a closed reductive group subscheme of type (RC) by
	\cite[XIX, 2.8 and XXII, \S 5]{sga3} containing
	\( L \), but it has trivial intersection with all root subgroups by theorem
	\ref{split-torus}. Theorem
	\ref{split-torus} also means that
	\( (T, \Phi) \) is an isotropic pinning in the sense of
	\cite{iso-elem} locally in Zariski topology (because the definition formally requires that all weight subspaces are free modules). In particular, by
	\cite[\S 2]{iso-elem} all root subgroups can be reconstructed by
	\( (T, \Phi) \) as follows. If
	\( \alpha \in \Phi \setminus 2 \Phi \), then
	\[
	U_\alpha = \bigcap_\lambda U_G(\lambda),
	\]
	where
	\( \lambda \) runs over all constant coweights of
	\( T \) such that
	\( \langle \alpha, \lambda \rangle > 0 \) and
	\[
	U_G(\lambda)(K')
	=
	\bigl\{
	g \in G(K')
	\mid
	\lim_{t \to 0} \up{ \lambda(t) }{g} = 1
	\bigr\}.
	\]
	If
	\( \alpha \in \Phi \cap 2 \Phi \), then
	\( U_\alpha \leq U_{\alpha / 2} \) is the locus where the action of the multiplicative monoid scheme
	\( \mathbb{A}^1 \) on
	\( U_{\alpha / 2} \) factors through the squaring morphism
	\( \mathbb{A}^1 \to \mathbb{A}^1 \).

	\section{%
		Reductive groups graded by
		\( \mathsf{A}_1 \)%
	}
	\label{a1-grading}
	
	We need an explicit construction of split simply connected isotropic reductive groups with
	\( \Phi \) of type
	\( \mathsf{A}_1 \). In such a group scheme there are only two root subgroups
	\( U^{\pm} \),
	\( P^{\pm} = U^{\pm} \rtimes L \) are opposite parabolic subgroups,
	\( L \) is their common Levi subgroup, and
	\( U^{\pm} \) are their unipotent radicals. The cojugation by an element
	\( w \in G(K) \) swaps
	\( P^{+} \) with
	\( P^{-} \) if and only if it swaps
	\( U^{+} \) with
	\( U^{-} \) and preserves
	\( L \). If such element exists, then the set of all such elements is an
	\( L(K) \)-coset because
	\( P^{\pm} \) are their own normalizers
	\cite[XXII, corollary 5.8.5]{sga3}.
	
	Our goal is to describe elements
	\( w \in U^{+}(K)\, U^{-}(K)\, U^{+}(K) \) conjugating
	\( P^{+} \) with
	\( P^{-} \) and vice versa. To obtain the equations on parameters of such
	\( w \) (with respect to fixed parameterizations of the root subgroups by
	\( K \)-modules) it is useful to find at least one such
	\( w_0 \), then get equations on
	\( w_0 L \), and finally impose the condition that generic
	\( w \in U^{+}(K)\, U^{-}(K)\, U^{+}(K) \) lies in this coset. This approach leads to significantly simpler calculations than the naive checking of
	\( \up{w}{P^{\pm}} = P^{\mp} \). In all cases the root subgroups
	\( U^{\pm} \) are parameterized by a certain
	\textit{Jordan pair} (or
	\textit{quadratic Jordan algebra} if we fix an element
	\( w_0 \) used above), and the Jordan operations explicitly appear in the formulae
	\( \up{w}{ t_{\pm}(p) } = t_{\mp}(f(p)) \). See
	\cite[\S 6.6]{loos-neher} for details and an even more general context.
	
	The only irreducible Tits indices with
	\( \Phi \) of type
	\( \mathsf{A}_1 \) are listed in table
	\ref{a1-list} together with the exception
	\( \TitsIndex{E}{}{7}{1}{78} \). The only coincidences between them are
	\begin{align*}
		\TitsIndex{A}{1}{1}{1}{(1)}
		&\cong
		\TitsIndex{B}{}{1}{1}{}
		\cong
		\TitsIndex{C}{}{1}{1}{(1)},
		&
		\TitsIndex{D}{1}{4}{1}{(1)}
		&\cong
		\TitsIndex{D}{1}{4}{1}{(4)}
		\sim
		\TitsIndex{D}{2}{4}{1}{(1)},
		\\
		\TitsIndex{B}{}{2}{1}{}
		&\cong
		\TitsIndex{C}{}{2}{1}{(2)},
		&
		\TitsIndex{A}{1}{2 d - 1}{1}{(d)}
		&\sim
		\TitsIndex{A}{2}{2 d - 1}{1}{(d)}
		\text{ for }
		d \geq 2,
		\\
		\TitsIndex{A}{1}{3}{1}{(2)}
		&\cong
		\TitsIndex{D}{1}{3}{1}{(1)}
		\sim
		\TitsIndex{A}{2}{3}{1}{(2)}
		\cong
		\TitsIndex{D}{2}{3}{1}{(1)}
		&
		\TitsIndex{D}{1}{n}{1}{(1)}
		&\sim
		\TitsIndex{D}{2}{n}{1}{(1)}
		\text{ for }
		n \geq 3.
	\end{align*}
	Here
	\( \cong \) means that two Tits indices differ by an exceptional isomorphism of Dynkin diagrams and, possibly, by an outer automorphism. The symbol
	\( \sim \) means that two Tits indices induce isomorphic maps
	\( u \). See
	\cite[\S 2]{diophantine} for the complete list of all such coincidences between irreducible Tits indices.
	
	\begin{table}[h!]
		\centering
		\begin{tabular}{|c|c|c|c|}
			\hline
			Tits index
			&
			restrictions
			&
			\( \widetilde{\Phi} \)
			&
			\( \deg(\e_i) \)
			\\ \hline \hline
			\( \TitsIndex{A}{1}{2 d - 1}{1}{(d)} \)
			&
			\( d \geq 1 \)
			&
			\(
			\{
			\e_i - \e_j
			\mid
			0 \leq i, j \leq 2 d - 1;\,
			i \neq j
			\}
			\)
			&
			\( - 1 / 2 \),
			\( i < d \)
			\\
			&&&
			\( 1 / 2 \),
			\( i \geq d \)
			\\ \hline
			\( \TitsIndex{A}{2}{2 d - 1}{1}{(d)} \)
			&
			\( d \geq 2 \)
			&
			\(
			\{
			\e_i - \e_j
			\mid
			0 \leq i, j \leq 2 d - 1;\,
			i \neq j
			\}
			\)
			&
			\( - 1 / 2 \),
			\( i < d \)
			\\
			&&&
			\( 1 / 2 \),
			\( i \geq d \)
			\\ \hline
			\( \TitsIndex{B}{}{n}{1}{} \)
			&
			\( n \geq 1 \)
			&
			\(
			\{
			\pm \e_i \pm \e_j
			\mid
			1 \leq i < j \leq n
			\}
			\)
			&
			\( 0 \),
			\( i < n \)
			\\
			&&
			\(
			{}
			\sqcup
			\{ \pm \e_i \mid 1 \leq i \leq n \}
			\)
			&
			\( 1 \),
			\( i = n \)
			\\ \hline
			\( \TitsIndex{C}{}{d}{1}{(d)} \)
			&
			\( d = 2^k \geq 1 \)
			&
			\(
			\{
			\pm \e_i \pm \e_j
			\mid
			1 \leq i < j \leq d
			\}
			\)
			&
			\( 1 / 2 \)
			\\
			&&
			\(
			{}
			\sqcup
			\{ \pm 2 \e_i \mid 1 \leq i \leq d \}
			\)
			&
			\\ \hline
			\( \TitsIndex{D}{1}{d}{1}{(d)} \)
			&
			\( d = 2^k \geq 4 \)
			&
			\(
			\{
			\pm \e_i \pm \e_j
			\mid
			1 \leq i < j \leq d
			\}
			\)
			&
			\( 1 / 2 \)
			\\ \hline
			\( \TitsIndex{D}{1}{n}{1}{(1)} \),
			\( \TitsIndex{D}{2}{n}{1}{(1)} \)
			&
			\( n \geq 3 \)
			&
			\(
			\{
			\pm \e_i \pm \e_j
			\mid
			1 \leq i < j \leq n
			\}
			\)
			&
			\( 0 \),
			\( i < n \)
			\\
			&&&
			\( 1 \),
			\( i = n \)
			\\ \hline
		\end{tabular}
		\caption{
			Classical Tits indices with
			\( \Phi = \mathsf{A}_1 \) given by the grading function
			\(
			\deg
			\colon
			\widetilde{\Phi}
			\to
			\{ - 1, 0, 1 \}
			\).
		}
		\label{a1-list}
	\end{table}
	
	The split group scheme of type
	\( \TitsIndex{A}{1}{2 d - 1}{1}{(d)} \) (and
	\( \TitsIndex{A}{2}{2 d - 1}{1}{(d)} \) for
	\( d \geq 2 \)) is the special linear group scheme
	\( G = \mathbb{SL}_{2 d} \). Let
	\( M_{-} = M_{+} = K^d \), so
	\[
	G(K) = \mathrm{SL}( M_{-} \oplus M_{+} ).
	\]
	Its distinguished subgroups are
	\begin{align*}
		P^{\pm}(K) &= \{ g \mid g(M_{\mp}) = M_{\mp} \},
		\\
		U^{\pm}(K)
		&=
		\{
		g
		\mid
		g|_{M_{\mp}} = 1,\,
		(g - 1)(M_{\pm}) \leq M_{\mp}
		\},
		\\
		L(K)
		&=
		\{ g \mid g(M_{-}) = M_{-},\, g(M_{+}) = M_{+} \},
		\\
		T(K)
		&=
		\{
		g
		\mid
		g|_{M_{\pm}}
		=
		\lambda^{\mp 1}
		\text{ for some }
		\lambda \in K^*
		\}.
	\end{align*}
	Root elements are
	\( t_{\pm}(p) = 1 + p \) for
	\(
	p
	\in
	\mathrm{Hom}( M_{\mp}, M_{\pm} )
	\cong
	\mat(d, K)
	\). An easy calculation shows that the conjugation by
	\( w = t_{+}(x)\, t_{-}(y)\, t_{+}(z) \) swaps
	\( P^{+} \) with
	\( P^{-} \) if and only if
	\( x = z \) is invertible and
	\( y = - x^{- 1} \). In this case
	\( w(m) = x m \) for
	\( m \in M_{+} \),
	\( w(m) = - x^{- 1}(m) \) for
	\( m \in M_{-} \),
	\[
	w L
	=
	\{ g \mid g(M_{+}) = M_{-},\, g(M_{-}) = M_{+} \},
	\]
	and
	\( w \) acts on root elements by
	\( \up{w}{ t_{+}(p) } = t_{-}( - x^{- 1} p x^{- 1} ) \),
	\( \up{w}{ t_{-}(p) } = t_{+}(- x p x) \).
	
	The split group scheme of type
	\( \TitsIndex{C}{}{d}{1}{(d)} \) is the group subscheme
	\( \mathbb{S} \mathrm{p}_{2 d} \leq \mathbb{SL}_{2 d} \), namely, the stabilizer of the symplectic form
	\[
	B\bigl( (x_{-}, x_{+}), (y_{-}, y_{+}) \bigr)
	=
	x_{-} \cdot y_{+} - y_{-} \cdot x_{+},
	\]
	where
	\( x_{\pm}, y_{\pm} \in M_{\pm} \) and
	\( x \cdot y = \sum_{i = 1}^d x_i y_i \) is the canonical pairing
	\( M_{-} \times M_{+} \to K \). The group subschemes
	\( P^{\pm} \),
	\( U^{\pm} \),
	\( L \) are the intersections of
	\( G \) with the corresponding group subschemes of
	\( \mathbb{SL}_{2 d} \), and the standard torus is the same. The root subgroups are parameterized by the submodule
	\( \{ p \mid p^{\mathrm{t}} = p \} \leq \mat(d, K) \)
	under the canonical isomorphisms
	\( \mathrm{Hom}( M_{\pm}, M_{\mp} ) \cong \mat(d, K) \). Again the conjugation by
	\( w = t_{+}(x)\, t_{-}(y)\, t_{+}(z) \) swaps
	\( P^{+} \) with
	\( P^{-} \) if and only if
	\( x = z \) is invertible and
	\( y = - x^{- 1} \), so the formulae for
	\( w \),
	\( w L \), and
	\( \up{w}{t_{\pm}(p)} \) are the same as for
	\( \mathbb{SL}_{2 d} \).
	
	In the cases
	\( \TitsIndex{B}{}{n}{1}{} \) and
	\(
	\TitsIndex{D}{1}{n}{1}{(1)}
	\sim
	\TitsIndex{D}{2}{n}{1}{(2)}
	\) let
	\( M_0 \) be the
	\( K \)-module with the split quadratic form
	\( q_0 \) of rank
	\( 2 n - 1 \) or
	\( 2 n - 2 \) respectively, i.e.
	\( q_0(x_1, x_2, \ldots) = x_1 x_2 + x_3 x_4 + \ldots \), and the last summand is
	\( x_{n - 1}^2 \) in the odd rank case. We embed
	\( M_0 \) to
	\( M = K e_{-} \oplus M_0 \oplus K e_{+} \) with the split quadratic form
	\( q(x_{-}, x_0, x_{+}) = x_{-} x_{+} + q_0(x_0) \). Let
	\( \mathrm{C}(M, q) \) be the
	\textit{Clifford algebra} of
	\( (M, q) \), i.e. the unital associative
	\( K \)-algebra generated by the
	\( K \)-module
	\( M \) with the relations
	\( m^2 = q(m) \) for
	\( m \in M \). This algebra is
	\( (\mathbb{Z} / 2 \mathbb{Z}) \)-graded, where the generating module
	\( M \) lies in the odd part. The group
	\( G(K) = \mathrm{Spin}(M, q) \) consists of invertible even elements
	\( g \in \mathrm{C}(M, q) \) such that
	\( \up{g}{M} = M \) and
	\( g^{- 1} = \alpha(g) \), where
	\( \alpha \) is the canonical involution on
	\( \mathrm{C}(M, q) \) acting trivially on
	\( M \). Its distinguished subgroups are
	\begin{align*}
		P^{\pm}(K)
		&=
		\{ g \mid \up{g}{e_{\mp}} \in K e_{\mp} \},
		\\
		U^{\pm}(K)
		&=
		\{
		g
		\mid
		g \in 1 + e_{\mp} \mathrm{C}(M_0, q_0)
		\},
		\\
		L(K)
		&=
		\{
		g
		\mid
		\up{g}{e_{-}} \in K e_{-},\,
		\up{g}{e_{+}} \in K e_{+}
		\},
		\\
		T(K)
		&=
		\{
		\lambda e_{-} e_{+} + \lambda^{- 1} e_{+} e_{-}
		\mid
		\lambda \in K^*
		\}.
	\end{align*}
	The root elements are
	\( t_{\pm}(p) = 1 + p e_{\mp} \) for
	\( p \in M_0 \). The conjugation by
	\( w = t_{+}(x)\, t_{-}(y)\, t_{+}(z) \) swaps
	\( P^{+} \) with
	\( P^{-} \) if and only if
	\( x = z \),
	\( q_0(x) \in K^* \), and
	\( y = q_0(x)^{- 1} x \), in this case
	\begin{align*}
		w &= x e_{-} + q_0(x)^{- 1} x e_{+},
		\\
		w L
		&=
		\{
		g
		\mid
		\up{g}{e_{-}} \in K e_{+},\,
		\up{g}{e_{+}} \in K e_{-}
		\},
		\\
		\up{w}{t_{+}(a)}
		&=
		t_{-}\bigl(
		q_0(x)^{- 2} B_0(x, a) x - q_0(x)^{- 1} a
		\bigr),
		\\
		\up{w}{t_{-}(a)}
		&=
		t_{+}\bigl( B_0(x, a) x - q_0(x) a \bigr),
	\end{align*}
	where
	\( B_0(x, y) = q_0(x + y) - q_0(x) - q_0(y) \) is the associated symmetric bilinear form.
	
	The split group scheme of type
	\( \TitsIndex{D}{1}{d}{1}{(d)} \) is again
	\( G = \mathbb{S} \mathrm{pin}_{2 d} \). The group subschemes
	\( P^{\pm} \),
	\( U^{\pm} \),
	\( L \) are the preimages in
	\( G \) of the corresponding group subschemes of
	\( \mathbb{SL}_{2 d} \), and the standard split torus is
	\[
	T(K)
	=
	\bigl\{
	\lambda^{d / 2}
	\prod_{i = 1}^d
	(\lambda^{- 1} e_i e_{- i} + e_{- i} e_i)
	\mid
	\lambda \in K^*
	\bigr\}.
	\]
	The root subgroups
	\( U^{\pm}(K) \) are parameterized by
	\(
	\{
	p \in \mat(d, K)
	\mid
	p^{\mathrm{t}} = - p,\,
	\forall i \enskip p_{i i} = 0
	\}
	\) as
	\begin{align*}
		t_{+}(p)
		&=
		\prod_{1 \leq i < j \leq d}
		(1 + p_{i j} e_{- i} e_{- j}),
		&
		t_{-}(p)
		&=
		\prod_{1 \leq i < j \leq d}
		(1 + p_{i j} e_i e_j),
	\end{align*}
	where
	\(
	M_{\pm}
	=
	\bigoplus_{i = 1}^d
	K e_{\pm i}
	\). The conjugation by
	\( w = t_{+}(x)\, t_{-}(y)\, t_{+}(z) \) swaps
	\( P^{+} \) with
	\( P^{-} \) if and only if
	\( x = z \) is invertible and
	\( y = - x^{- 1} \), this follows by considering the images in
	\( \mathrm{SO}(M, q) \leq \mathrm{SL}_{2 d}(K) \). Of course, the formulae
	\( \up{w}{ t_{\pm}(p) } \) are the same as in the case of
	\( \mathbb{SL}_{2 d} \).
	
	Finally, there is the exceptional case
	\( \TitsIndex{E}{}{7}{1}{78} \). We need several definitions and constructions from nonassociative algebra. Recall that
	\( C \) is a
	\textit{composition algebra}
	\cite[\S 19]{albert-alg} over
	\( K \) if it is a finitely generated projective
	\( K \)-module with bilinear multiplication, there is the identity element
	\( 1 \in C \) with respect to the multiplication, and there is a regular (or semiregular in the odd rank case) quadratic form
	\( q \colon C \to K \) satisfying
	\( q(1) = 1 \) and the composition law
	\( q(x y) = q(x)\, q(y) \). Let
	\( \langle x, y \rangle = q(x + y) - q(x) - q(y) \) be the associated symmetric bilinear form and
	\( \inv x = \langle x, 1 \rangle\, 1 - x \). Then
	\begin{align*}
		\inv{\inv x} &= x,
		&
		\langle x y, z \rangle
		&=
		\langle y, \inv{x} z \rangle
		=
		\langle x, z \inv{y} \rangle,
		\\
		q(\inv x) &= q(x),
		&
		x^2 y &= x (x y),
		\\
		\inv{\,1\,} &= 1,
		&
		x y^2 &= (x y) y,
		\\
		\inv{x} \inv{y} &= \inv{(y x)},
		&
		x (\inv{x} y) &= q(x)\, y,
		\\
		x\, \inv{x} &= \inv{x}\, x = q(x)\, 1,
		&
		(x y) \inv{y} &= q(y)\, x.
	\end{align*}
	
	\'Etale locally every composition algebra is either
	\( K \) itself with
	\( q(x) = x^2 \) and
	\( \inv x = x \), or
	\( K \times K \) with
	\( q(x, y) = x y \) and
	\( \inv{(x, y)} = (y, x) \), or
	\( \mat(2, K) \) with
	\( q = \det \) and
	\( \inv{ \sMat{x}{y}{z}{w} } = \sMat{w}{- y}{- z}{x} \), or the
	\textit{Zorn algebra}
	\( \mathrm{Z}(K) = \sMat{K}{K^3}{K^3}{K} \) with the multiplication
	\[
	\sMat{\alpha}{x}{y}{\beta}
	\sMat{\alpha'}{x'}{y'}{\beta'}
	=
	\sMat
	{ \alpha \alpha' + x \cdot y' }
	{ \alpha x' + x \beta' - y \times y' }
	{ y \alpha' + \beta y' + x \times x' }
	{ \beta \beta' + y \cdot x' },
	\]
	the identity
	\( \sMat{1}{0}{0}{1} \), the quadratic form
	\(
	q \sMat{\alpha}{x}{y}{\beta}
	=
	\alpha \beta - x \cdot y
	\), and the involution
	\(
	\inv{ \sMat{\alpha}{x}{y}{\beta} }
	=
	\sMat{\beta}{- x}{- y}{\alpha}
	\), where
	\( x \cdot y \) is the dot product and
	\( x \times y \) is the cross product of three dimensional vector-columns over
	\( K \). Fix the Zorn algebra
	\( C = \mathrm{Z}(K) \).
	
	A
	\textit{cubic norm structure} is a
	\( K \)-module
	\( A \) with a distinguished element
	\( 1 \in A \), a cubic form
	\( \Norm \colon A \to K \), a quadratic map
	\( ({-})^\# \colon A \to A \) with the polarization
	\( x \times y = (x + y)^\# - x^\# - y^\# \), and a symmetric bilinear form
	\( \Trace( {-}, {=} ) \colon A \times A \to K \) such that
	\begin{align*}
		1^\# &= 1,
		&
		1 \times x &= \Trace(1, x)\, 1 - x,
		\\
		\Norm(1) &= 1,
		&
		\Norm(x + y)
		&=
		\Norm(x)
		+ \Trace(x^\#, y)
		+ \Trace(x, y^\#)
		+ \Norm(y),
		\\
		x^{\# \#} &= \Norm(x)\, x,
		&
		\Norm(x^\#) &= \Norm(x)^2.
	\end{align*}
	The last axiom is omitted in the standard sources
	\cite[\S 33]{albert-alg} because
	\( A \) is usually assumed to be a finitely generated projective module and in this case the axiom follows from the previous ones. Every cubic norm stricture satisfies additional identities
	\begin{align*}
		(x, y, z)
		&=
		\Trace(x \times y, z)
		=
		\Trace(x, y \times z)
		\text{ is symmetric},
		\\
		\Trace(1, x \times y)
		&=
		\Trace(1, x) \Trace(1, y) - \Trace(x, y),
		\\
		x^\# \times (x \times y)
		&=
		\Trace(x^\#, y)\, x + \Norm(x)\, y,
		\\
		x \times (x^\# \times y)
		&=
		\Trace(x, y)\, x^\# + \Norm(x)\, y,
		\\
		(x \times y)^\# + x^\# \times y^\#
		&=
		\Trace(x, y^\#)\, x + \Trace(x^\#, y)\, y,
		\\
		\Trace(x^\#, x) &= 3 \Norm(x),
		\\
		\Norm(x \times y)
		&=
		\Trace(x^\#, y) \Trace(x, y^\#)
		-
		\Norm(x) \Norm(y),
		\\
		x^\# \times x
		&=
		\Trace(1, x) \Trace(1, x^\#)\, 1
		- \Norm(x)\, 1
		- \Trace(1, x^\#)\, x
		- \Trace(1, x)\, x^\#,
	\end{align*}
	see
	\cite[table of identities 33a]{albert-alg} for details. Every cubic norm structure is also a
	\textit{quadratic Jordan algebra}
	\cite[theorem 33.9]{albert-alg}, namely,
	\[
	\Bimul_x(y)
	=
	\Trace(x, y)\, x - x^\# \times y.
	\]
	An element
	\( x \in A \) is invertible in the sense of quadratic Jordan algebras if and only if
	\( \Norm(x) \in K^* \), in this case
	\( x^{- 1} = \Norm(x)^{- 1}\, x^\# \)
	\cite[corollary 33.10]{albert-alg}.
	
	We need only one cubic norm structure, the split
	\textit{Albert algebra}
	\cite[theorem 36.5, \S 39.18, \S 39.20]{albert-alg}
	\[
	A
	=
	\{
	a \in \mat(3, C)
	\mid
	a_{i i} \in K,\,
	a_{i j} = \inv{ a_{j i} }
	\}.
	\]
	The operations are given by
	\begin{align*}
		1
		&=
		\bigl(
		\begin{smallmatrix}
			1 & 0 & 0
			\\
			0 & 1 & 0
			\\
			0 & 0 & 1
		\end{smallmatrix}
		\bigr),
		\\
		\Bigl(
		\begin{smallmatrix}
			\alpha_1 & c_3 & \inv{c_2}
			\\
			\inv{c_3} & \alpha_2 & c_1
			\\
			c_2 & \inv{c_1} & \alpha_3
		\end{smallmatrix}
		\Bigr)^\#
		&=
		\Bigl(
		\begin{smallmatrix}
			\alpha_2 \alpha_3 - q(c_1)
			&
			\inv{(c_1 c_2)} - \alpha_3 c_3
			&
			c_3 c_1 - \alpha_2 \inv{c_2}
			\\
			c_1 c_2 - \alpha_3 \inv{c_3}
			&
			\alpha_1 \alpha_3 - q(c_2)
			&
			\inv{(c_2 c_3)} - \alpha_1 c_1
			\\
			\inv{(c_3 c_1)} - \alpha_2 c_2
			&
			c_2 c_3 - \alpha_1 \inv{c_1}
			&
			\alpha_1 \alpha_2 - q(c_3)
		\end{smallmatrix}
		\Bigr),
		\\
		\Norm\Bigl(
		\begin{smallmatrix}
			\alpha_1 & c_3 & \inv{c_2}
			\\
			\inv{c_3} & \alpha_2 & c_1
			\\
			c_2 & \inv{c_1} & \alpha_3
		\end{smallmatrix}
		\Bigr)
		&=
		\alpha_1 \alpha_2 \alpha_3
		+ \langle 1, c_1 c_2 c_3 \rangle
		- \alpha_1\, q(c_1)
		- \alpha_2\, q(c_2)
		- \alpha_3\, q(c_3),
		\\
		\Trace\Bigl(
		\Bigl(
		\begin{smallmatrix}
			\alpha_1 & c_3 & \inv{c_2}
			\\
			\inv{c_3} & \alpha_2 & c_1
			\\
			c_2 & \inv{c_1} & \alpha_3
		\end{smallmatrix}
		\Bigr),
		\Bigl(
		\begin{smallmatrix}
			\beta_1 & d_3 & \inv{d_2}
			\\
			\inv{d_3} & \beta_2 & d_1
			\\
			d_2 & \inv{d_1} & \beta_3
		\end{smallmatrix}
		\Bigr)
		\Bigr)
		&=
		\alpha_1 \beta_1
		+ \alpha_2 \beta_2
		+ \alpha_3 \beta_3
		+ \langle c_1, d_1 \rangle
		+ \langle c_2, d_2 \rangle
		+ \langle c_3, d_3 \rangle.
	\end{align*}
	
	Following
	\cite{muh-wei-e7} let
	\( W = K \times A \times K \times A \times K \). This set is a
	\( 2 \)-step nilpotent
	\( K \)-module with the nilpotent filtration
	\( 0 \leq 0 \times 0 \times K \times 0 \times 0 \leq W \) and the operations
	\begin{align*}
		(r, b, s, c, t) \dotplus (r', b', s', c', t')
		&=
		(
		r + r',
		b + b',
		s + s' - t r' - \Trace(c, b'),
		c + c',
		t + t'
		),
		\\
		(r, b, s, c, t) \cdot k
		&=
		(r k, b k, s k^2, c k, t k),
		\\
		\tau(r, b, s, c, t)
		&=
		(0, 0, 2 s + r t + \Trace(b, c), 0, 0).
	\end{align*}
	Let
	\[
	\Theta(r, b, s, c, t)
	=
	s^2
	+
	s \Trace(b, c)
	+
	r s t
	+
	\Trace(b^\#, c^\#)
	-
	r \Norm(c)
	+
	t \Norm(b),
	\]
	this is a form of degree
	\( 4 \) with respect to the action of the multiplicative monoid
	\( K^\bullet \), i.e.
	\( \Theta(w \cdot k) = \Theta(w)\, k^4 \). We define
	\( G(K) \) to be the group of all automorphisms of
	\( W \) as a
	\( 2 \)-step nilpotent
	\( K \)-module preserving
	\( \Theta \) and acting trivially on
	\( 0 \times 0 \times K \times 0 \times 0 \). Let
	\begin{align*}
		P^{+}(K)
		&=
		\bigl\{
		g
		\mid
		g( K \times 0 \times 0 \times 0 \times 0 )
		=
		K \times 0 \times 0 \times 0 \times 0
		\bigr\},
		\\
		P^{-}(K)
		&=
		\bigl\{
		g
		\mid
		g( 0 \times 0 \times 0 \times 0 \times K )
		=
		0 \times 0 \times 0 \times 0 \times K
		\bigr\},
		\\
		L(K)
		&=
		\bigl\{
		g
		\mid
		g(r, b, s, c, t)
		=
		(g_1(r), g_2(b), s, g_4(c), g_5(t))
		\text{ for some maps } g_i
		\bigr\},
		\\
		U^{+}(K)
		&=
		\bigl\{
		g
		\mid
		g(1, 0, 0, 0, 0) = (1, 0, 0, 0, 0),\,
		g(0, b, 0, 0, 0)
		\in
		K \times \{ b \} \times 0 \times 0 \times 0
		\bigr\},
		\\
		U^{-}(K)
		&=
		\bigl\{
		g
		\mid
		g(0, 0, 0, 0, 1) = (0, 0, 0, 0, 1),\,
		g(0, 0, 0, c, 0)
		\in
		0 \times 0 \times 0 \times \{ c \} \times K
		\bigr\}
	\end{align*}
	be subgroups of
	\( G(K) \) and
	\begin{align*}
		t_{+}(a)
		&\colon
		(r, b, s, c, t)
		\mapsto
		\bigl(
		r
		+
		\Trace(a, b)
		+
		\Trace(a^\#, c)
		-
		\Norm(a)\, t
		,
		b + a \times c - t a^\#
		, \\
		&\qquad\qquad\qquad\qquad
		s
		-
		\Trace(a, c^\#)
		+
		\Trace(a^\#, c)\, t
		-
		\Norm(a)\, t^2
		,
		c - t a
		,
		t
		\bigr),
		\\
		t_{-}(a)
		&\colon
		(r, b, s, c, t)
		\mapsto
		\bigl(
		r
		,
		b + r a
		,
		s
		-
		\Trace(a, b^\#)
		-
		\Trace(a^\#, b)\, r
		-
		\Norm(a)\, r^2
		, \\
		&\qquad\qquad\qquad\qquad
		c + a \times b + a^\# r
		,
		t
		-
		\Trace(a, c)
		-
		\Trace(a^\#, b)
		-
		\Norm(a)\, r
		\bigr),
		\\
		d(u)
		&\colon
		(r, b, s, c, t)
		\mapsto
		( u^3 r, u b, s, u^{- 1} c, u^{- 3} t ),
		\\
		[g]
		&\colon
		(r, b, s, c, t)
		\mapsto
		( r, g(b), s, g^{- \dagger}(c), t )
	\end{align*}
	be elements of
	\( G(K) \) for
	\( a \in A \),
	\( \alpha \in K^* \), and
	\( g \in \mathrm{E}^{\mathrm{sc}}_6(K) \). Here we use that
	\( \mathrm{E}^{\mathrm{sc}}_6(K) \) is canonically isomorphic to the group of linear automorphisms of
	\( A \) perserving
	\( \Norm \)
	\cite[theorem 58.2]{albert-alg} and
	\( g^\dagger \) denotes the adjoint operator with respect to the form
	\( \Trace \), so
	\( g^{\dagger \dagger} = g \),
	\( g(x^\#) = g^{- \dagger}(x)^\# \),
	\( g^{- \dagger} \in \mathrm{E}^{\mathrm{sc}}_6(K) \), and
	\( g \mapsto g^{- \dagger} \) is a canonical representative of the outer automorphism of
	\( \mathbb{E}^{\mathrm{sc}}_6 \).
	
	\begin{theorem}
		\label{e7-sc}
		The group scheme
		\( G \) is split reductive of type
		\( \mathsf{E}_7^{\mathrm{sc}} \),
		\( P^{\pm} \) are its opposite parabolic subgroups corresponding to
		\( \TitsIndex{E}{}{7}{1}{78} \)-grading,
		\( L = P^{+} \cap P^{-} \) is their common Levi subgroup, and
		\( U^{\pm} \) are their unipotent radicals. The group subscheme
		\( L \) is the central product of
		\( \mathbb{E}_6^{\mathrm{sc}} \) and
		\( \mathbb{G}_{\mathrm{m}} \) by their common central group subscheme
		\( \mu_3 \). The standard split torus is
		\( T(K) = \{ d(\lambda) \mid \lambda \in K^* \} \). The root subgroups
		\( U^{\pm} \) are parameterized by the morphisms
		\( t_{\pm} \). The conjugation by
		\( w = t_{+}(x)\, t_{-}(y)\, t_{+}(z) \) swaps
		\( P^{+} \) with
		\( P^{-} \) if and only if
		\( x = z \),
		\( \Norm(x) \in K^* \), and
		\( y = - x^{- 1} \) (recall that
		\( x^{- 1} = \Norm(x)^{- 1}\, x^\# \)), in this case
		\begin{align*}
			w(r, b, s, c, t)
			&=
			\bigl(
			- \Norm(x)\, t,
			- \Norm(x) \Bimul_{x^{- 1}}(c),
			s + \Trace(b, c) + r t,
			\Norm(x)^{- 1} \Bimul_x(b),
			\Norm(x)^{- 1}\, r
			\bigr),
			\\
			w L
			&=
			\bigl\{
			g \in G(K)
			\mid
			g(r, b, s, c, t)
			=
			(
			g_1(t),
			g_2(c),
			s + \Trace(b, c) + r t,
			g_4(b),
			g_5(r)
			)
			\text{ for some maps }
			g_i
			\bigr\},
			\\
			\up{w}{ t_{+}(a) }
			&=
			t_{-}\bigl( - \Bimul_{x^{- 1}}(a) \bigr),
			\\
			\up{w}{ t_{-}(a) }
			&=
			t_{+}\bigl( - \Bimul_x(a) \bigr).
		\end{align*}
	\end{theorem}
	\begin{proof}
		It is not hard to check directly that
		\( t_{\pm} \colon A \to U^{\pm}(K) \) are group isomorphisms and the group scheme
		\( L \) is the central product of images of
		\( d \colon \mathbb{G}_{\mathrm{m}} \to L \) and
		\( [{-}] \colon \mathbb{E}^{\mathrm{sc}}_6 \to L \) with common central subgroup
		\( \mu_3 \). Moreover,
		\( P^{\pm}(K) = U^{\pm}(K) \rtimes L(K) \), and the group
		\( L(K) \) acts on
		\( U^{\pm}(K) \) by
		\begin{align*}
			\up{d(u)}{ t_{+}(a) } &= t_{+}(u^2 a),
			&
			\up{d(u)}{ t_{-}(a) } &= t_{-}(u^{-2} a),
			\\
			\up{[g]}{ t_{+}(a) }
			&=
			t_{+}( g^{- \dagger}(a) ),
			&
			\up{[g]}{ t_{-}(a) } &= t_{-}(g(a)).
		\end{align*}
		Also,
		\( U^{-}(K) \cap P^{+}(K) = \{ 1 \} \) and
		\( g \in U^{-}(K)\, P^{+}(K) = \Omega(K) \) if and only if the first coordinate of
		\( g(1, 0, 0, 0, 0) \) is invertible.
		
		It follows that
		\( G \) is an affine group scheme of finite presentation with the smooth open subscheme
		\( \Omega \cong U^{-} \times L \times U^{+} \), so its fiberwise connected component
		\( G^\circ \) is a smooth group subscheme by
		\cite[VI\textsubscript{B}, theorem 3.10]{sga3}. On the other hand, the geometric fibers of
		\( G \) are simple simply connected algebraic groups of type
		\( \mathsf{E}_7 \) by
		\cite[remark 10.10]{muh-wei-e7}, so
		\( G \) is a reductive group scheme of type
		\( \mathsf{E}_7^{\mathrm{sc}} \). The standard maximal torus of
		\( \mathbb{E}_6^{\mathrm{sc}} \) consists of all operators acting diagonally in the standard basis of
		\( A \), and its central product with the other factor of
		\( L \) gives a maximal torus of
		\( G \). This torus clearly splits. All our constructions make sense over the base ring
		\( \mathbb{Z} \) with connected spectrum and trivial Picard group, so
		\( G \) also splits. The only
		\( 3 \)-grading on the root system
		\( \mathsf{E}_7 \) corresponds to the Tits index
		\( \TitsIndex{E}{}{7}{1}{78} \), so the group subschemes from the statement form the required isotropic structure. The image of
		\( d \) is the center of
		\( L \), so it is the standard split torus.
		
		It remains to prove the claims concerning
		\( w \). First of all, there is
		\[
		w_0
		=
		t_{+}(1)\, t_{-}(- 1)\, t_{+}(1)
		\colon
		(r, b, s, c, t)
		\mapsto
		(- t, - c, s + \Trace(b, c) + r t, b, r)
		\]
		such that the conjugation by
		\( w_0 \) swaps
		\( P^{+} \) and
		\( P^{-} \). The equations for
		\( w_0 L \) follow immediately. If the conjugation by
		\( w = t_{+}(x)\, t_{-}(y)\, t_{+}(z) \) swaps
		\( P^{+} \) and
		\( P^{-} \), then
		\( t_{+}(x)\, t_{-}(y) \in w_0 P^{+} \), so
		\( \Norm(x) \in K^* \) and
		\( y = - x^{- 1} \). Similarly,
		\( z = - y^{- 1} = x \). The final formulae for
		\( w(r, b, s, c, t) \) and
		\( \up{w}{ t_{\pm}(a) } \) follow by direct calculations.
	\end{proof}

	\section{Normalizers of long root subgroups}
	
	In this section we also need the list of all irreducible Tits indices with
	\( \Phi \) of type
	\( \mathsf{BC}_1 \). The classical ones are given in table
	\ref{bc1-list} and the exceptional ones are
	\( \TitsIndex{D}{3}{4}{1}{9} \),
	\( \TitsIndex{D}{6}{4}{1}{9} \),
	\( \TitsIndex{E}{2}{6}{1}{35} \),
	\( \TitsIndex{E}{}{7}{1}{66} \),
	\( \TitsIndex{E}{}{8}{1}{133} \),
	\( \TitsIndex{F}{}{4}{1}{21} \),
	\( \TitsIndex{E}{2}{6}{1}{29} \),
	\( \TitsIndex{E}{}{7}{1}{48} \),
	\( \TitsIndex{E}{}{8}{1}{91} \). The only coincidences are
	\begin{align*}
		\TitsIndex{A}{2}{3}{1}{(1)}
		&\cong
		\TitsIndex{D}{2}{3}{1}{(2)},
		&
		\TitsIndex{D}{1}{4}{1}{(2)}
		&\sim
		\TitsIndex{D}{2}{4}{1}{(2)}
		\sim
		\TitsIndex{D}{3}{4}{1}{9}
		\sim
		\TitsIndex{D}{6}{4}{1}{9},
		&
		\TitsIndex{D}{1}{n}{1}{(d)}
		&\sim
		\TitsIndex{D}{2}{n}{1}{(d)}
		\text{ for }
		2 \leq d = 2^k \mid 2 n
		\text{ and }
		d + 2 \leq n.
	\end{align*}
	
	\begin{table}[h!]
		\centering
		\begin{tabular}{|c|c|c|c|}
			\hline
			Tits index
			&
			restrictions
			&
			\( \widetilde{\Phi} \)
			&
			\( \deg(\e_i) \)
			\\ \hline \hline
			\( \TitsIndex{A}{2}{n}{1}{(d)} \)
			&
			\( 1 \leq d \mid n + 1 \)
			&
			\(
			\{
			\e_i - \e_j
			\mid
			0 \leq i, j \leq n;\,
			i \neq j
			\}
			\)
			&
			\( - 1 \),
			\( i < d \)
			\\
			&
			\( 2 d \leq n \)
			&&
			\( 0 \),
			\( d \leq i \leq n - d \)
			\\
			&&&
			\( 1 \),
			\( i > n - d \)
			\\ \hline
			\( \TitsIndex{C}{}{n}{1}{(d)} \)
			&
			\( 2 \leq d = 2^k \mid 2 n \)
			&
			\(
			\{
			\pm \e_i \pm \e_j
			\mid
			1 \leq i < j \leq n
			\}
			\)
			&
			\( 0 \),
			\( i \leq n - d \)
			\\
			&
			\( d + 1 \leq n \)
			&
			\(
			{}
			\sqcup
			\{ \pm 2 \e_i \mid 1 \leq i \leq n \}
			\)
			&
			\( 1 \),
			\( i > n - d \)
			\\ \hline
			\( \TitsIndex{D}{1}{n}{1}{(d)} \)
			&
			\( 2 \leq d = 2^k \mid 2 n \)
			&
			\(
			\{
			\pm \e_i \pm \e_j
			\mid
			1 \leq i < j \leq n
			\}
			\)
			&
			\( 0 \),
			\( i \leq n - d \)
			\\
			&
			\( d + 2 \leq n \)
			&&
			\( 1 \),
			\( i > n - d \)
			\\ \hline
			\( \TitsIndex{D}{2}{n}{1}{(d)} \)
			&
			\( 2 \leq d = 2^k \mid 2 n \)
			&
			\(
			\{
			\pm \e_i \pm \e_j
			\mid
			1 \leq i < j \leq n
			\}
			\)
			&
			\( 0 \),
			\( i \leq n - d \)
			\\
			&
			\( d + 1 \leq n \)
			&&
			\( 1 \),
			\( i > n - d \)
			\\ \hline
		\end{tabular}
		\caption{
			Classical Tits indices with
			\( \Phi = \mathsf{BC}_1 \) given by the grading function
			\(
			\deg
			\colon
			\widetilde{\Phi}
			\to
			\{ - 2, - 1, 0, 1, 2 \}
			\).
		}
		\label{bc1-list}
	\end{table}
	
	\begin{lemma}
		\label{non-deg-3}
		Let
		\( \mathfrak{g} \) be the Lie algebra of adjoint split reductive group scheme of type
		\( \widetilde{\Phi} \) with the
		\( 3 \)-grading
		\(
		\mathfrak{g}
		=
		\mathfrak{g}_{- 1}
		\oplus
		\mathfrak{g}_0
		\oplus
		\mathfrak{g}_1
		\) induced by an irreducible Tits index with
		\( \Phi \) of type
		\( \mathsf{A}_1 \) except
		\(
		\TitsIndex{A}{1}{1}{1}{(1)}
		\cong
		\TitsIndex{B}{}{1}{1}{}
		\cong
		\TitsIndex{C}{}{1}{1}{(1)}
		\). Then the maps
		\(
		\mathfrak{g}_{- 1}
		\to
		\Hom( \mathfrak{g}_1, \mathfrak{g}_0 )
		\) and
		\(
		\mathfrak{g}_0
		\to
		\Hom( \mathfrak{g}_1, \mathfrak{g}_1 )
		\) induced by the Lie bracket are injective.
	\end{lemma}
	\begin{proof}
		Let
		\(
		\widetilde{\Phi}
		=
		\widetilde{\Phi}_{- 1}
		\sqcup
		\widetilde{\Phi}_0
		\sqcup
		\widetilde{\Phi}_1
		\) be the corresponding decomposition of the root system. We are going to prove that for every
		\( \alpha \in \Phi_{- 1} \) or
		\( \alpha \in \Phi_0 \) there is
		\( \beta \in \Phi_1 \) such that
		\( \alpha + \beta \) is a root and
		\( \beta \not{\perp} \alpha \). This claim implies that
		\(
		\mathfrak{g}_{- 1}
		\to
		\Hom( \mathfrak{g}_1, \mathfrak{g}_0 )
		\) is injective and the kernel of
		\(
		\mathfrak{g}_0
		\to
		\Hom( \mathfrak{g}_1, \mathfrak{g}_1 )
		\) is contained in the Lie algebra of the standard maximal torus. But then this kernel is trivial because the group scheme is adjoint and
		\( \widetilde{\Phi}_1 \) spans the root lattice.
		
		The claim itself is easy to check for
		\( \TitsIndex{E}{}{7}{1}{78} \) using the corresponding root diagram
		\cite{atlas}. For classical cases this follows from tables
		\ref{a1-m2z} and
		\ref{a1-z2p}, where we list all possible roots
		\( \alpha \) up to the action of the Weyl group of
		\( \widetilde{\Phi}_0 \).
	\end{proof}
	
	\begin{table}[h!]
		\centering
		\begin{tabular}{|c|c|c|c|}
			\hline
			Tits index
			&
			restrictions
			&
			\( \alpha \)
			&
			\( \beta \)
			\\ \hline \hline
			\( \TitsIndex{A}{1}{2 d - 1}{1}{(d)} \)
			&
			\( d \geq 2 \)
			&
			\( \e_0 - \e_d \)
			&
			\( \e_d - \e_1 \)
			\\ \hline
			\( \TitsIndex{B}{}{n}{1}{} \)
			&
			\( n \geq 2 \)
			&
			\( - \e_n \)
			&
			\( \e_n - \e_1 \)
			\\
			&&
			\( \e_1 - \e_n \)
			&
			\( \e_n \)
			\\ \hline
			\( \TitsIndex{C}{}{d}{1}{(d)} \)
			&
			\( d = 2^k \geq 2 \)
			&
			\( - \e_1 - \e_2 \)
			&
			\( 2 \e_1 \)
			\\
			&&
			\( - 2 \e_1 \)
			&
			\( \e_1 + \e_2 \)
			\\ \hline
			\( \TitsIndex{D}{1}{d}{1}{(d)} \)
			&
			\( d = 2^k \geq 4 \)
			&
			\( - \e_1 - \e_2 \)
			&
			\( \e_2 + \e_3 \)
			\\ \hline
			\( \TitsIndex{D}{1}{n}{1}{(1)} \)
			&
			\( n \geq 3 \)
			&
			\( \e_1 - \e_n \)
			&
			\( \e_n - \e_2 \)
			\\ \hline
		\end{tabular}
		\caption{
			Injectivity of
			\(
			\mathfrak{g}_{- 1}
			\to
			\Hom( \mathfrak{g}_1, \mathfrak{g}_0 )
			\) from lemma
			\ref{non-deg-3}.
		}
		\label{a1-m2z}
	\end{table}
	
	\begin{table}[h!]
		\centering
		\begin{tabular}{|c|c|c|c|}
			\hline
			Tits index
			&
			restrictions
			&
			\( \alpha \)
			&
			\( \beta \)
			\\ \hline \hline
			\( \TitsIndex{A}{1}{2 d - 1}{1}{(d)} \)
			&
			\( d \geq 2 \)
			&
			\( \e_0 - \e_1 \)
			&
			\( \e_d - \e_0 \)
			\\
			&&
			\( \e_d - \e_{d + 1} \)
			&
			\( \e_{d + 1} - \e_0 \)
			\\ \hline
			\( \TitsIndex{B}{}{n}{1}{} \)
			&
			\( n \geq 2 \)
			&
			\( \e_1 \)
			&
			\( \e_n - \e_1 \)
			\\
			&&
			\( \e_1 - \e_2 \)
			&
			\( \e_n - \e_1 \)
			\\ \hline
			\( \TitsIndex{C}{}{d}{1}{(d)} \)
			&
			\( d = 2^k \geq 2 \)
			&
			\( \e_1 - \e_2 \)
			&
			\( 2 \e_2 \)
			\\ \hline
			\( \TitsIndex{D}{1}{d}{1}{(d)} \)
			&
			\( d = 2^k \geq 4 \)
			&
			\( \e_1 - \e_2 \)
			&
			\( \e_2 + \e_3 \)
			\\ \hline
			\( \TitsIndex{D}{1}{n}{1}{(1)} \)
			&
			\( n \geq 3 \)
			&
			\( \e_1 - \e_2 \)
			&
			\( \e_n - \e_1 \)
			\\ \hline
		\end{tabular}
		\caption{
			Injectivity of
			\(
			\mathfrak{g}_0
			\to
			\Hom( \mathfrak{g}_1, \mathfrak{g}_1 )
			\) from lemma
			\ref{non-deg-3}
		}
		\label{a1-z2p}
	\end{table}
	
	\begin{lemma}
		\label{non-deg-5}
		Let
		\( \mathfrak{g} \) be the Lie algebra of adjoint split reductive group scheme of type
		\( \widetilde{\Phi} \) with the
		\( 5 \)-grading
		\(
		\mathfrak{g}
		=
		\bigoplus_{i = - 2}^2
		\mathfrak{g}_i
		\) induced by an irreducible Tits index with
		\( \Phi \) of type
		\( \mathsf{BC}_1 \) and
		\(
		\widetilde{\Phi}
		=
		\bigsqcup_{i = - 2}^2
		\widetilde{\Phi}_i
		\) be the corresponding decomposition of the absolute root system. Let also
		\( \mathfrak{g}_{0'} \leq \mathfrak{g}_0 \) be the sum of all absolute root subspaces with roots from
		\(
		\widetilde{\Phi}_0
		\cap
		( \widetilde{\Phi}_{- 2} + \widetilde{\Phi}_2 )
		\) (this is always a component or two components of
		\( \widetilde{\Phi}_0 \)). We consider
		\( \mathfrak{g}_{0'} \) as a direct summand in
		\( \mathfrak{g}_0 \) with the complement containing the Lie algebra of the standard torus and other root subspaces. If
		\( \widetilde{\Phi}_2 \) had at least two roots, then the maps
		\(
		\mathfrak{g}_{- 1}
		\to
		\Hom( \mathfrak{g}_2, \mathfrak{g}_1 ),\,
		x \mapsto ( y \mapsto [x, y] )
		\) and
		\(
		\mathfrak{g}_{- 1}
		\to
		\Hom( \mathfrak{g}_1, \mathfrak{g}_{0'} ),\,
		x \mapsto ( y \mapsto \pi( [x, y] ) )
		\) are injective, where
		\( \pi \colon \mathfrak{g}_0 \to \mathfrak{g}_{0'} \) is the projection onto a direct summand.
	\end{lemma}
	\begin{proof}
		As in lemma
		\ref{non-deg-3}, we only have to show that for every
		\( \alpha \in \widetilde{\Phi}_{- 1} \) there are
		\( \beta \in \widetilde{\Phi}_2 \) and
		\( \gamma \in \widetilde{\Phi}_1 \) such that
		\( \alpha + \beta \) and
		\( \alpha + \gamma \) are roots,
		\( \alpha \not{\perp} \beta \),
		\( \alpha \not{\perp} \gamma \), and
		\(
		\alpha + \gamma
		\in
		\widetilde{\Phi}_{- 2} + \widetilde{\Phi}_2
		\). For exceptional Tits indices use the corresponding root diagrams
		\cite{atlas} and for classical ones see table
		\ref{bc1-add}, where we listed all possibilities for
		\( \alpha \) up to the action of the Weyl group of
		\( \widetilde{\Phi}_0 \).
	\end{proof}
	
	\begin{table}[h!]
		\centering
		\begin{tabular}{|c|c|c|c|c|}
			\hline
			Tits index
			&
			restrictions
			&
			\( \alpha \)
			&
			\( \beta \)
			&
			\( \gamma \)
			\\ \hline \hline
			\( \TitsIndex{A}{2}{n}{1}{(d)} \)
			&
			\( 2 \leq d \mid n + 1 \),
			\( 2 d \leq n \)
			&
			\( \e_d - \e_n \)
			&
			\( \e_n - \e_0 \)
			&
			\( \e_{n - 1} - \e_d \)
			\\
			&&
			\( \e_0 - \e_d \)
			&
			\( \e_n - \e_0 \)
			&
			\( \e_d - \e_1 \)
			\\ \hline
			\( \TitsIndex{C}{}{n}{1}{(d)} \)
			&
			\( 2 \leq d = 2^k \mid 2 n \),
			\( d + 1 \leq n \)
			&
			\( \e_1 - \e_n \)
			&
			\( 2 \e_n \)
			&
			\( \e_{n - 1} - \e_1 \)
			\\ \hline
			\( \TitsIndex{D}{2}{n}{1}{(d)} \)
			&
			\( 4 \leq d = 2^k \mid 2 n \),
			\( d + 2 \leq n \)
			&
			\( \e_1 - \e_n \)
			&
			\( \e_n + \e_{n - 1} \)
			&
			\( \e_{n - 1} - \e_1 \)
			\\ \hline
		\end{tabular}
		\caption{
			Injectivity of
			\(
			\mathfrak{g}_{- 1}
			\to
			\Hom( \mathfrak{g}_2, \mathfrak{g}_1 )
			\) and
			\(
			\mathfrak{g}_{- 1}
			\to
			\Hom( \mathfrak{g}_1, \mathfrak{g}_{0'} )
			\) from lemma
			\ref{non-deg-5}.
		}
		\label{bc1-add}
	\end{table}
	
	\begin{theorem}
		\label{long-norm}
		Let
		\( G \) be an isotropic reductive group scheme over
		\( K \),
		\( \alpha \in \Phi \) be a long root, and
		\( X \subseteq P_\alpha = \mathfrak{g}_\alpha \) be a generating set of a
		\( K \)-module. Assume that the Tits index is not
		\(
		\TitsIndex{A}{1}{1}{1}{(1)}
		\cong
		\TitsIndex{B}{}{1}{1}{}
		\cong
		\TitsIndex{C}{}{1}{1}{(1)}
		\) or
		\( G \) is simply connected. Then the scheme normalizer
		\[
		N
		=
		\{
		g
		\mid
		\up{g}{ t_\alpha(X) } \subseteq U_\alpha
		\}
		\]
		is the parabolic group subscheme generated as an fppf sheaf by
		\( L \) and all
		\( U_\beta \) with
		\( \angle(\alpha, \beta) \leq \pi / 2 \).
	\end{theorem}
	\begin{proof}
		This is
		\cite[theorem 1]{diophantine} if the rank of
		\( \Phi \) is at least
		\( 2 \), so we can assume that
		\( \Phi \) is of type
		\( \mathsf{A}_1 \) or
		\( \mathsf{BC}_1 \). Clearly, the parabolic group subscheme is contained in
		\( N \), so we only have to check that every
		\( g \in N(K) \) lies in the parabolic subgroup. Without loss of generality,
		\( G \) splits and
		\( g \) lies in
		\( U_{- \alpha}(K) \) or
		\( U_{- \alpha / 2}(K) \) (e.g. by Gauss decomposition
		\cite[lemma 1]{diophantine}), so it remains to check that
		\( g = 1 \). The case of simply connected
		\( G \) with the Tits index
		\(
		\TitsIndex{A}{1}{1}{1}{(1)}
		\cong
		\TitsIndex{B}{}{1}{1}{}
		\cong
		\TitsIndex{C}{}{1}{1}{(1)}
		\) also follows from
		\cite[theorem 1]{diophantine} because this group scheme embeds into
		\( \mathbb{SL}_3 \) preserving root subgroups. Other cases with long root Lie subalgebras of rank
		\( 1 \) follow from the same theorem because the isotropic structure can be refined to an isotropic structure of larger rank preserving root subgroups,
		\begin{align*}
			\TitsIndex{A}{2}{n}{1}{(1)}
			&\mapsto
			\TitsIndex{A}{1}{n}{n}{(1)}
			\text{ for }
			n \geq 2,
			\\
			\TitsIndex{D}{2}{n}{1}{(2)}
			&\mapsto
			\TitsIndex{D}{1}{n}{n}{(1)}
			\text{ for }
			n \geq 3
			\text{ %
				(this also covers
				\( \TitsIndex{D}{1}{n}{1}{(2)} \) for
				\( n \geq 4 \) and
				\(
				\TitsIndex{D}{3}{4}{1}{9}
				\sim
				\TitsIndex{D}{6}{4}{1}{9}
				\))%
			},
			\\
			\TitsIndex{E}{2}{6}{1}{35}
			&\mapsto
			\TitsIndex{E}{2}{6}{2}{16'}
			\qquad
			\TitsIndex{E}{}{7}{1}{66}
			\mapsto
			\TitsIndex{E}{}{7}{2}{31},
			\qquad
			\TitsIndex{E}{}{8}{1}{133}
			\mapsto
			\TitsIndex{E}{}{8}{2}{66}.
		\end{align*}
		
		Consider the adjoint representation of
		\( G \) on the Lie algebra
		\( \mathfrak{g} \) of the group scheme
		\( G / \Cent(G) \). If
		\( \Phi \) is of type
		\( \mathsf{A}_1 \), then this Lie algebra is
		\( 3 \)-graded and
		\( \up{g}{ t_\alpha(x) } \in U_\alpha(K) \) for every
		\( x \in X \) by assumption, so
		\( \up{g}{ t_\alpha(x) } \) centralizes
		\( \mathfrak{g}_1 \). Take any
		\(
		z
		\in
		g^{- 1}(\mathfrak{g}_1)
		\), it is an element centralized by all
		\( t_\alpha(x) \). By lemma
		\ref{non-deg-3} necessarily
		\( z \in \mathfrak{g}_1 \), so
		\( g \) centralizes
		\( \mathfrak{g}_1 \). But then the same lemma means that
		\( g = 1 \).
		
		It remains to consider the case with the relative root system of type
		\( \mathsf{BC}_1 \). The long root subgroups generate an isotropic group subscheme as an fppf sheaf, so we can apply lemma
		\ref{non-deg-3} to this group subscheme and we already know that
		\( g \) is trivial if only
		\( g \in U_{- \alpha}(K) \). Take any
		\(
		z
		\in
		g^{- 1}(\mathfrak{g}_1 \oplus \mathfrak{g}_2)
		\), it is again centralized by all
		\( t_\alpha(x) \). By lemmas
		\ref{non-deg-3} and
		\ref{non-deg-5} necessarily
		\( z_{- 2} = 0 \),
		\( z_{- 1} = 0 \), and
		\( z_{0'} = 0 \). Here
		\( \mathfrak{g}_{0'} \leq \mathfrak{g}_0 \) is the direct summand from lemma
		\ref{non-deg-5} and
		\( z_{- 2} \in \mathfrak{g}_{- 2} \),
		\( z_{- 1} \in \mathfrak{g}_{- 1} \),
		\( z_{0'} \in \mathfrak{g}_{0'} \) are the corresponding components of
		\( z \). Lemma
		\ref{non-deg-5} applied to
		\( z \in g^{- 1}(\mathfrak{g}_1) \) then implies
		\( g \in U_{- \alpha}(K) \), so
		\( g = 1\).
	\end{proof}
	
	\begin{theorem}
		\label{root-norm}
		Let
		\( G \) be an isotropic reductive group scheme over
		\( K \) and assume that the Tits index is not
		\(
		\TitsIndex{A}{1}{1}{1}{(1)}
		\cong
		\TitsIndex{B}{}{1}{1}{}
		\cong
		\TitsIndex{C}{}{1}{1}{(1)}
		\) or
		\( G \) is simply connected. Then
		\[
		L
		=
		\{
		g
		\mid
		\forall \alpha \enskip
		\up{g}{U_\alpha} = U_\alpha
		\}.
		\]
	\end{theorem}
	\begin{proof}
		This easily follows from theorem
		\ref{long-norm} because the intersection of all parabolic group subschemes is still a group subscheme of type (RC)
		\cite[lemma 2]{diophantine}, but corresponding to the empty set of relative roots.
	\end{proof}

	\section{Weyl elements}
	
	\begin{lemma}
		\label{ssys-sgr}
		Let
		\( G \) be an isotropic reductive group scheme over
		\( K \) and
		\( \Psi \subseteq \Phi \) be a closed root subsystem, i.e. a subset with the properties
		\( \Psi = - \Psi \) and
		\( (\Psi + \Psi) \cap \Phi \subseteq \Psi \). Consider the fppf group subsheaf
		\(
		H
		=
		\langle
		U_\alpha
		\mid
		\alpha \in \Psi
		\rangle
		\). Then
		\( H \) is a closed reductive group subscheme of
		\( G \). Moreover,
		\( H \) is isotropic with the distinguished group subschemes
		\( L \cap H \) and
		\( U_\alpha \) for
		\( \alpha \in \Psi \) in the following generalized sense. Fppf locally there is a splitting of
		\( H \) such that these group subschemes are generated in the scheme sense by disjoint families of root subgroups (and by the maximal torus in the case of
		\( L \cap H \)) and the corresponding partition of every component of the absolute root system is described by an irreducible Tits index.
	\end{lemma}
	\begin{proof}
		Without loss of generality,
		\( G \) splits. The group subscheme
		\( H \) is semisimple reductive by
		\cite[\S XXII.5]{sga3}, namely, it is the scheme derived subgroup of the group subscheme
		\( H \widetilde{T} \) of type (RC). It corresponds to the absolute root subsystem
		\(
		\widetilde{\Psi}
		=
		u^{- 1}(\Psi)
		\subseteq
		\widetilde{\Phi}
		\). Moreover,
		\( H \) also splits with the maximal torus
		\( H \cap \widetilde{T} \), and
		\( U_\alpha \) are generated by its absolute root subgroups for
		\( \alpha \in \Psi \) (i.e. with roots from
		\( \widetilde{\Psi} \)). The intersection
		\( L_H = L \cap H \) is the centralizer of
		\( T \) in
		\( H \), so it is a reductive group scheme, and it is of type (RC) because
		\( H \cap \widetilde{T} \leq L_H \). Clearly, it is generated by
		\( H \cap \widetilde{T} \) and remaining absolute root subgroups.
		
		We claim that the partition of every component
		\( \widetilde{\Psi} \) induced by
		\( L_H \) and
		\( U_\alpha \) corresponds to an irreducible Tits index. This can be checked directly by considering all possible irreducible Tits indices of
		\( G \) and all maximal irreducible closed root subsystems of
		\( \Phi \). Alternatively, recall that every irreducible Tits index arises from a maximal isotropic structure of a simple reductive group scheme
		\( G' \) over a field
		\( \Bbbk \). The group subscheme
		\( L' \leq G' \) is anisotropic, i.e. it does not contain a non-central split torus, so
		\( L'_H = L' \cap H' \) and its image in
		\( H' / \Cent(H') \) are also anisotropic according to
		\cite[XXVI, corollary 6.12]{sga3}. Take the maximal split torus
		\( T'_{\mathrm{spl}} \leq \Cent(L') \). It is a maximal split torus in
		\( G' \),
		\( L' \) is its centralizer, and
		\( U'_\alpha \) are the corresponding root subgroups by
		\cite[XXVI, corollary 6.16 and 7.4.2]{sga3}. Then
		\( ( T'_{\mathrm{spl}} \cap H' )^\circ \) is a maximal split torus in
		\( H' \) because every split torus in
		\( H' \) has finite intersection with the ``complementary'' split torus
		\( \Cent_{ T'_{\mathrm{spl}} }(H') \). Here
		\( ({-})^\circ \) is the maximal subtorus of a group scheme of multiplicative type. The image of this torus in
		\( H' / \Cent(H') \) is maximal by
		\cite[proposition 4.25]{borel-tits}.
		
		Now consider the canonical decomposition
		\[
		H' / \Cent(H')
		=
		\prod_{i = 1}^n
		\mathrm{R}_{\Bbbk_i / \Bbbk}(H'_i)
		\]
		\cite[XXIV, proposition 5.10]{sga3}, where
		\( \Bbbk_i / \Bbbk \) are finite separable extensions,
		\( \mathrm{R}_{\Bbbk_i / \Bbbk} \) are Weil restrictions, and
		\( H'_i \) are simple reductive group schemes over
		\( \Bbbk_i \). Clearly,
		\( L'_H / \Cent(H') \), the image of
		\( ( T'_{\mathrm{spl}} \cap H' )^\circ \), and the isomorphic images of
		\( U'_\alpha \) in
		\( H' / \Cent(H') \) for
		\( \alpha \in \Psi \) decompose into products of group subschemes of
		\( \mathrm{R}_{\Bbbk_i / \Bbbk}(H'_i) \). The factors of
		\( L'_H / \Cent(H') \) and the images of
		\( U'_\alpha \) are themselves Weil restrictions of group subschemes of
		\( H'_i \), denote them by
		\( L'_{H, i} \) and
		\( U'_{\alpha, i} \). The factors of the image of
		\( ( T'_{\mathrm{spl}} \cap H' )^\circ \) are the maximal split tori of Weil restrictions of some maximal split tori
		\( T'_{ \mathrm{spl}, i } \leq L'_{H, i} \) by
		\cite[6.19 and 6.20]{borel-tits}, in particular, every
		\( L'_{H, i} \) is the centralizer of
		\( T'_{ \mathrm{spl}, i } \) in
		\( H'_i \). By construction
		\( L'_{H, i} \) are the root subgroups with respect to
		\( T'_{ \mathrm{spl}, i } \).
	\end{proof}
	
	Now we are ready to prove main results about Weyl elements.
	
	\begin{theorem}
		\label{weyl}
		Let
		\( G \) be an isotropic reductive group scheme over
		\( K \) and assume that the Tits index is not
		\(
		\TitsIndex{A}{1}{1}{1}{(1)}
		\cong
		\TitsIndex{B}{}{1}{1}{}
		\cong
		\TitsIndex{C}{}{1}{1}{(1)}
		\) or that
		\( G \) is simply connected. Then we have the following.
		\begin{enumerate}
			
			\item
			Every
			\( \alpha \)-Weyl element
			\(
			w
			\in
			U_\alpha(K)\, U_{- \alpha}(K)\, U_\alpha(K)
			\) in
			\( G(K) \) normalizes
			\( L \) and the standard split torus
			\( T \) from theorem
			\ref{split-torus}, it acts on
			\( T \) via the corresponding Euclidean reflection. Also,
			\( \up{w}{U_\beta} = U_{ s_\alpha(\beta) } \) as schemes.
			
			\item
			If
			\( g_1 \in U_\alpha(K) \) is a factor of a Weyl element, i.e. there are
			\( g_2 \in U_{- \alpha}(K) \) and
			\( g_3 \in U_\alpha(K) \) such that
			\( g_1 g_2 g_3 \) is
			\( \alpha \)-Weyl, then
			\( g_2 \) and
			\( g_3 \) are unique. In other words,
			\( G(K) \) has
			\textit{unique Weyl extensions} in the sense of
			\cite[\S 2.2.C]{wiedemann}. This is already known for isotropic rank at least
			\( 2 \) by
			\cite[proposition 5.20]{loos-neher}.
			
			\item
			The locus of
			\( x \in P_\alpha \) such that
			\( t_\alpha(x) \) is a factor of a Weyl element is an open subscheme
			\( P_\alpha^* \subset P_\alpha \) with non-empty fibers. In particular, all Weyl elements exists locally in Zariski topology.
			
		\end{enumerate}
	\end{theorem}
	\begin{proof}
		Let
		\( \alpha \in \Phi \setminus \frac{1}{2} \Phi \) be a non-ultrashort root. Firstly we check that fppf locally there is an
		\( \alpha \)-Weyl element satisfying the first property. Indeed, let
		\( H = \langle L, U_{- \alpha}, U_\alpha \rangle \) be a group subscheme of type (RC), it is graded by
		\( \mathsf{A}_1 \) by lemma
		\ref{ssys-sgr}. Using results from
		\S \ref{a1-grading} including theorem
		\ref{e7-sc} we see that fppf locally there is an
		\( \alpha \)-Weyl element
		\( w \) such that
		\( \up{w}{ U_{\pm \alpha} } = U_{\mp \alpha} \) and
		\( \up{w}{L_H} = L_H \), where
		\( L_H = L \cap H \) is a part of the isotropic structure on
		\( H \). The torus
		\( T \) decomposes into the product of
		\( T_H \) and
		\( \Ker(\alpha) \) in the sense of fppf sheaves, where
		\( T_H \) is the canonical split torus of
		\( H \) and
		\( \alpha \) is considered as a homomorphism
		\( T \to \mathbb{G}_{\mathrm{m}} \). Clearly,
		\( w \) centralizes
		\( \Ker(\alpha) \) and inverts
		\( T_H \), so it normalizes
		\( T \) and preserves the root system
		\( \Phi \). It follows that
		\( w \) acts on
		\( T \) by the reflection
		\( s_\alpha \), so it preserves
		\( L \) and permutes
		\( U_\beta \) in the required way. In particular,
		\( T_H \leq T \) is the centralizer of the orthogonal complement of
		\( \alpha \) in the weight lattice.
		
		Now we prove the theorem. If
		\( w \in G(K) \) is an
		\( \alpha \)-Weyl element, then fppf locally choose an
		\( \alpha \)-Weyl element
		\( w_0 \) as above. The product
		\( w w_0 \) normalizes all root subgroups
		\( U_\beta(K) \), so it lies in
		\( L(K) \) by theorem
		\ref{long-norm} (as in the proof of theorem
		\ref{root-norm}). It follows that
		\( w \) satisfies the first claim.
		
		Let
		\( N \leq G \) be the group subscheme generated by
		\( L \) and all Weyl elements fppf locally. It is easy to see that
		\( N \) is a closed smooth affine group subscheme of
		\( G \),
		\( L \) is its normal group subscheme, and
		\( L / G \) is canonically isomorphic to the Weyl group of
		\( \Phi \). So as a scheme
		\( N \) is isomorphic to the disjoint union of
		\( | \mathrm{W}(\Phi) | \) copies of
		\( L \). Let us denote these copies by
		\( \dot{w} L \) for
		\( w \in \mathrm{W}(\Phi) \).
		
		An element
		\( g_1 \in U_\alpha(K) \) is a factor of an
		\( \alpha \)-Weyl element if and only if there are
		\( g_2 \in U_{- \alpha}(K) \),
		\( g_3 \in U_\alpha(K) \) such that
		\( g_1 g_2 g_3 \in ( \dot{s}_\alpha L )(K) \), i.e.
		\[
		g_1
		\in
		U_{- \alpha}(K)\,
		(\dot{s}_\alpha L)(K)\,
		U_{- \alpha}(K).
		\]
		But the product
		\(
		U_{- \alpha}
		\times
		\dot{s}_\alpha L
		\times
		U_{- \alpha}
		\) embeds to the reductive group scheme
		\( \langle U_{- \alpha}, L, U_\alpha \rangle \) as an open subscheme (after multiplication by an
		\( \alpha \)-Weyl element from the right this is the product
		\( U_{- \alpha} \times L \times U_\alpha \)), so the locus of such
		\( g_1 \) is open and the factors
		\( g_2 \),
		\( g_3 \) are unique. The locus also has non-empty fibers because all Weyl elements exist fppf locally.
	\end{proof}
	
	It turns out that
	\( \Phi \) as a subset of the character lattice of
	\( T \) naturally completes to a root datum.
	
	\begin{lemma}
		\label{coroots}
		Let
		\( G \) be an isotropic reductive group scheme over
		\( K \). Then for every root
		\( \alpha \) there is a well defined
		\textit{coroot}
		\(
		\alpha^{\vee}
		\colon
		\mathbb{G}_{\mathrm{m}}
		\to
		T
		\) such that
		\(
		\langle \beta, \alpha^\vee \rangle
		=
		2 (\alpha \cdot \beta) / (\alpha \cdot \alpha)
		\) for all roots
		\( \beta \). Coroots form the dual root system
		\( \Phi^{\vee} \). If
		\( G \) is simply connected and
		\( \Phi \) is of type
		\( \mathsf{A}_1 \), then
		\( \alpha^{\vee} \) is an isomorphism. If
		\( G \to G' \) is a central isogeny, then coroots of
		\( G \) map to coroots of
		\( G' \). If
		\( \Psi \subset \Phi \) is a closed root subsystem and
		\( \alpha \in \Psi \), then coroots of the reductive group subscheme
		\( \langle U_\beta \mid \beta \in \Psi \rangle \) coincide with the corresponding coroots of
		\( G \).
	\end{lemma}
	\begin{proof}
		If
		\( \Phi \) is of type
		\( \mathsf{A}_1 \), then this follows from results of
		\S \ref{a1-grading} including theorem
		\ref{e7-sc}.
		
		In general to prove that a coroot exist we can assume that
		\( 2 \alpha \notin \Phi \). Let
		\( H = \langle U_{- \alpha}, U_\alpha \rangle \) be the reductive group subscheme with
		\( \mathsf{A}_1 \)-grading as in the proof of theorem
		\ref{weyl}. We have the coroot
		\(
		\alpha^{\vee}
		\colon
		\mathbb{G}_{\mathrm{m}}
		\to
		T_H
		\), where
		\( T_H \) is the standard split torus in
		\( H \). By theorem
		\ref{weyl} and its proof,
		\( \alpha^{\vee} \) spans the eigenspace of
		\( s_\alpha \) in the cocharacter lattice of
		\( T \) with the eigenvalue
		\( - 1 \). It follows that the 
		\( \alpha^{\vee} \) is the required coroot. All other properties are now obvious.
	\end{proof}
	
	\begin{theorem}
		\label{squares}
		Let
		\( G \) be an isotropic reductive group scheme over
		\( K \) and assume that the Tits index is not
		\(
		\TitsIndex{A}{1}{1}{1}{(1)}
		\cong
		\TitsIndex{B}{}{1}{1}{}
		\cong
		\TitsIndex{C}{}{1}{1}{(1)}
		\) or that
		\( G \) is simply connected. Let also
		\( \alpha \in \Phi \setminus \frac{1}{2} \Phi \) be a non-ultrashort root. Then every
		\( \alpha \)-Weyl element
		\( w_\alpha \) has the square
		\( \alpha^{\vee}(- 1) \), in particular,
		\[
		\up{ w_\alpha^2 }{ t_\beta(x) }
		=
		\begin{cases}
			t_\beta(x),
			&
			2
			(\alpha \cdot \beta)
			/
			(\alpha \cdot \alpha)
			\text{ is even},
			\\
			t_\beta(- x),
			&
			2
			(\alpha \cdot \beta)
			/
			(\alpha \cdot \alpha)
			\text{ is odd and }
			2 \beta \notin \Phi,
			\\
			t_\beta(x \cdot (- 1)),
			&
			2
			(\alpha \cdot \beta)
			/
			(\alpha \cdot \alpha)
			\text{ is odd and } 2 \beta \in \Phi
		\end{cases}
		\]
		for every
		\( \beta \in \Phi \).
	\end{theorem}
	\begin{proof}
		Without loss of generality,
		\( G \) splits. Let
		\( H = \langle U_{- \alpha}, U_\alpha \rangle \) be the isotropic reductive group subscheme with the relative root system
		\( \mathsf{A}_1 \) by lemma
		\ref{ssys-sgr}. We calculate
		\( w_\alpha^2 \) in the simply connected reductive finite central covering of
		\( H \) using results from
		\S \ref{a1-grading}.
		
		In the case of
		\( \TitsIndex{A}{1}{2 d - 1}{1}{(d)} \) for
		\( d \geq 1 \) and
		\( \TitsIndex{A}{2}{2 d - 1}{1}{(d)} \) for
		\( d \geq 2 \) we have
		\(
		w_\alpha
		=
		t_\alpha(x)\,
		t_{- \alpha}( - x^{- 1} )\,
		t_\alpha(x)
		\) for canonical choices of the parameter modules
		\( P_{\pm \alpha} \), so
		\( w_\alpha^2 = \alpha^{\vee}(- 1) = - 1 \) in the standard representation
		\( M_{-} \oplus M_{+} \). For
		\( \TitsIndex{C}{}{d}{1}{(d)} \) for
		\( d = 2^k \geq 1 \) we again have
		\(
		w_\alpha
		=
		t_\alpha(x)\,
		t_{- \alpha}( - x^{- 1} )\,
		t_\alpha(x)
		\) and
		\( w_\alpha^2 = \alpha^{\vee}(- 1) = - 1 \) in the standard representation. In the cases
		\( \TitsIndex{B}{}{n}{1}{} \) for
		\( n \geq 1\) and
		\(
		\TitsIndex{D}{1}{n}{1}{(1)}
		\sim
		\TitsIndex{D}{2}{n}{1}{(2)}
		\) for
		\( n \geq 3 \) we have
		\( w_\alpha = x e_{-} + q_0(x)^{- 1} x e_{+} \) in the Clifford algebra, so
		\( w_\alpha^2 = \alpha^{\vee}(- 1) = - 1 \). In the exceptional case
		\( \TitsIndex{E}{}{7}{1}{78} \) we use theorem
		\ref{e7-sc}, so
		\(
		w_\alpha
		=
		t_{+}(x)\, t_{-}( - x^{- 1} )\, t_{+}(x)
		\) has the square
		\[
		w_\alpha^2
		=
		d(- 1)
		=
		\alpha^{\vee}(- 1)
		\colon
		(r, b, s, c, t)
		\mapsto
		(- r, - b, s, - c, - t).
		\]
		
		Finally, in the case
		\( \TitsIndex{D}{1}{d}{1}{(d)} \) for
		\( d = 2^k \geq 4 \) we do not have a nice formula for
		\( w_\alpha \) as an element of the Clifford algebra. Fortunately, in this case
		\( L \) acts transitively on the scheme of Weyl elements because all symplectic matrices are congruent Zariski locally. Take the Weyl element
		\(
		w_\alpha
		=
		t_{+}(x)\, t_{-}( - x^{- 1} )\, t_{+}(x)
		\) with
		\(
		x
		=
		- x^{- 1}
		=
		e_{1 2} - e_{2 1} + e_{3 4} - e_{4 3} + \ldots
		\in
		\mat(d, K)
		\). We see that
		\(
		w_\alpha^2
		=
		\prod_{i = 1}^d
		( e_{- i} e_i - e_i e_{- i} )
		=
		\alpha^{\vee}(- 1)
		\) is central, so the same formula holds for all
		\( \alpha \)-Weyl elements.
	\end{proof}
	
	Theorem
	\ref{squares} does not hold for ultrashort roots. For example, consider the split simply connected isotropic reductive group scheme with the Tits index
	\( \TitsIndex{A}{2}{2}{1}{(1)} \), it has root elements
	\begin{align*}
		t_{+}(x, y, z)
		&=
		\Bigl(
		\begin{smallmatrix}
			1 & x & y
			\\
			0 & 1 & z
			\\
			0 & 0 & 1
		\end{smallmatrix}
		\Bigr),
		&
		t_{-}(x, y, z)
		&=
		\Bigl(
		\begin{smallmatrix}
			1 & 0 & 0
			\\
			x & 1 & 0
			\\
			y & z & 1
		\end{smallmatrix}
		\Bigr).
	\end{align*}
	The product
	\(
	w
	=
	t_{+}(x, y, z)\,
	t_{-}(x', y', z')\,
	t_{+}(x'', y'', z'')
	\) is a Weyl element if and only if
	\begin{align*}
		x y z - y^2 &\in K^*,
		\\
		x' &= z / (y - x z),
		&
		y' &= 1 / (x z - y),
		&
		z' &= - x / y,
		\\
		x'' &= x (x z - y) / y,
		&
		y'' &= y,
		&
		z'' &= y z / (x z - y),
	\end{align*}
	in this case
	\begin{align*}
		w
		&=
		\biggl(
		\begin{smallmatrix}
			0 & 0 & y
			\\
			0 & (y - x z) / y &0
			\\
			1 / (x z - y) & 0 & 0
		\end{smallmatrix}
		\biggr),
		&
		w^2
		&=
		\biggl(
		\begin{smallmatrix}
			y / (x z - y) & 0 & 0
			\\
			0 & (x z - y)^2 / y^2 & 0
			\\
			0 & 0 & y / (x z - y)
		\end{smallmatrix}
		\biggr).
	\end{align*}
	
	The case of adjoint isotropic reductive group schemes with the Tits index
	\(
	\TitsIndex{A}{1}{1}{1}{(1)}
	\cong
	\TitsIndex{B}{}{1}{1}{}
	\cong
	\TitsIndex{C}{}{1}{1}{(1)}
	\) is indeed exceptional in theorems
	\ref{long-norm},
	\ref{root-norm},
	\ref{weyl}, and
	\ref{squares}. Every such group scheme
	\( G \) is the special orthogonal group scheme of the quadratic form
	\( q(x, y, z) = x z + y^2 \) on the module
	\( M^\vee \oplus K \oplus M \), where
	\( M \) is a finitely generated projective module of constant rank
	\( 1 \),
	\( M^\vee \) is the dual module, and the products
	\( M \times M^\vee \to K \),
	\( M^\vee \times M \to K \) denote the canonical pairing. The distinguished subgroups and root elements of
	\( G(K) \) are
	\begin{align*}
		L(K) = T(K)
		&=
		\Bigl\{
		\Bigl(
		\begin{smallmatrix}
			\lambda & 0 & 0
			\\
			0 & 1 & 0
			\\
			0 & 0 & \lambda^{- 1}
		\end{smallmatrix}
		\Bigr)
		\mid
		\lambda \in K^*
		\Bigr\},
		\\
		t_{+}(x)
		&=
		\Bigl(
		\begin{smallmatrix}
			1 & - 2 x & - x^{\otimes 2}
			\\
			0 & 1 & x
			\\
			0 & 0 & 1
		\end{smallmatrix}
		\Bigr)
		\text{ for } x \in M,
		\\
		t_{-}(x)
		&=
		\Bigl(
		\begin{smallmatrix}
			1 & 0 & 0
			\\
			- x & 1 & 0
			\\
			- x^{\otimes 2} & 2 x & 1
		\end{smallmatrix}
		\Bigr)
		\text{ for } x \in M^\vee.
	\end{align*}
	Here we use that the endomorphism ring of our module is just the generalized matrix ring
	\[
	\Bigl(
	\begin{smallmatrix}
		K & M & M^{\otimes 2}
		\\
		M^{\vee} & K & M
		\\
		M^{{\vee} \otimes 2} & M^\vee & K
	\end{smallmatrix}
	\Bigr).
	\]
	If
	\( X \subseteq M \) is a generating subset, then the group scheme
	\( N = \{ g \mid \up{g}{ t_{+}(x) } \in U_{+} \} \) decomposes into the semidirect product
	\[
	U^{+}
	\rtimes
	\{ t_{-}(a) \mid a^{\otimes 2} = 2 a = 0 \}
	\rtimes
	L
	\]
	and coincides with the scheme normalizer of
	\( U^{+} \). Indeed, if
	\( t_{-}(a) \) conjugates
	\( t_{+}(X) \) into
	\( U_{+}(K) \), then
	\( a (a X) = 2 a X = 0 \), i.e.
	\( a^{\otimes 2} = 2 a = 0 \). The group subscheme
	\( N \) is generated by
	\( P^{+} \) and such root elements Zariski locally by the Gauss decomposition. In particular,
	\( N = P^{+} \) as a scheme if and only if
	\( 2 \in K^* \), and over any field
	\( N(\Bbbk) = P^{+}(\Bbbk) \). If follows that all Weyl elements have the form
	\(
	w
	=
	t_{+}(x + u)\,
	t_{-}( - x^{- 1} )\,
	t_{+}(x + v)
	\), where
	\( x \) is invertible,
	\( u^{\otimes 2} = 2 u = v^{\otimes 2} = 2 v = 0 \), so they exist if and only if
	\( G \) splits. Finally,
	\begin{align*}
		w
		&=
		\Bigl(
		\begin{smallmatrix}
			0 & 0 & - x^2
			\\
			u x^{- 2} & - 1 & v
			\\
			- x^{- 2} & 0 & 0
		\end{smallmatrix}
		\Bigr),
		&
		w^2
		&=
		\Bigl(
		\begin{smallmatrix}
			1 & 0 & 0
			\\
			(u + v) x^{- 2} & 1 & u + v
			\\
			0 & 0 & 1
		\end{smallmatrix}
		\Bigr).
	\end{align*}

	\bibliographystyle{plain}
	\bibliography{references}

\end{document}